\documentclass{amsart}
\usepackage[
top=30truemm,
bottom=30truemm,
left=30truemm,
right=30truemm]{geometry}
\usepackage{amsfonts, amssymb, amsmath, verbatim, amsthm, mathrsfs, amscd, enumerate, ascmac, fancyhdr, caption, hyperref, framed,ulem,subfigure}
\usepackage{bbm}
\usepackage{eucal}
\usepackage{tikz}
\usetikzlibrary{shapes.geometric}
\usetikzlibrary{arrows.meta,arrows}
\usepackage{graphicx} 
\usepackage{pgf, tikz}
\usepackage{pgfplots}
\pgfplotsset{compat=1.18}
\usepackage{float}
\usepackage{pdfpages}

\usepackage[framemethod=tikz]{mdframed}
\mdfsetup{linecolor=blue,backgroundcolor=gray!10,roundcorner=10pt,leftmargin=2pt,innerleftmargin=2pt,innerrightmargin=2pt}
\usepackage{url}

\usepackage{xcolor}
\usepackage{tikz}
\usetikzlibrary{positioning,intersections, calc, arrows,decorations.markings}

\newtheorem{theorem}{Theorem}[section]
\newtheorem{lemma}[theorem]{Lemma}
\newtheorem{proposition}[theorem]{Proposition}
\newtheorem*{theorem*}{Theorem}

\newtheorem*{conjecture*}{Conjecture}

\newtheorem*{claim*}{Claim}
\theoremstyle{definition}

\newtheorem{example}[theorem]{Example}

\newtheorem*{goal*}{Goal}

\numberwithin{equation}{section}
\numberwithin{theorem}{section}

\theoremstyle{remark}
\newtheorem{remark}[theorem]{Remark}

\def\supp{{\rm supp}}
\def\sp{{\rm span}}

\title{Mizohata--Takeuchi inequalities for orthonormal systems}

\author[Bennett]{Jonathan Bennett}
\address[Jonathan Bennett]{School of Mathematics, The Watson Building, University of Birmingham, Edgbaston, Birmingham, B15 2TT, England.}
\email{J.Bennett@bham.ac.uk}

\author[Bez]{Neal Bez}
\address[Neal Bez]{Graduate School of Mathematics, Nagoya University, Nagoya 464-0862, Japan and Graduate School of Science and Engineering, Saitama University, Saitama 338-8570, Japan}
\curraddr{Graduate School of Mathematical Sciences, The University of Tokyo, 3-8-1 Komaba, Meguro-ku, Tokyo 153-8914, Japan}
\email{bez@ms.u-tokyo.ac.jp}

\author[Guti\'errez]{Susana Guti\'errez}
\address[Susana Guti\'errez]{School of Mathematics, The Watson Building, University of Birmingham, Edgbaston, Birmingham, B15 2TT, England.}
\email{S.Gutierrez@bham.ac.uk}

\author[Nakamura]{Shohei Nakamura}
\address[Shohei Nakamura]{School of Mathematics, The Watson Building, University of Birmingham, Edgbaston, Birmingham, B15 2TT, England.}
\email{s.nakamura@bham.ac.uk}

\author[Oliveira]{Itamar Oliveira}
\address[Itamar Oliveira]{School of Mathematics, The Watson Building, University of Birmingham, Edgbaston, Birmingham, B15 2TT, England.}
\email{i.oliveira@bham.ac.uk, oliveira.itamar.w@gmail.com}

\begin{document}

\date{\today}

\begin{abstract}
We establish some weighted $L^2$ inequalities for Fourier extension operators in the setting of orthonormal systems. In the process we develop a direct approach to such inequalities based on generalised Wigner distributions, complementing the Schatten space approach that is prevalent in the wider context of estimates for such orthonormal systems. 
Our results are set within a broader family of tentatively suggested ($L^p$) inequalities of Mizohata--Takeuchi type. For $p$ an even integer we see that such weighted inequalities may be recast as questions of co-positivity of tensor forms, and for $p\leq 1$ we provide some evidence that they may hold in reverse provided the orthonormal sequence is complete.
\end{abstract}

\maketitle
\tableofcontents
\section{Introduction}
\subsection{Mizohata--Takeuchi inequalities}
A basic and enduring objective of harmonic analysis is to quantify or describe the effects of interference in wave fields arising from superpositions of plane waves propagating in different directions in euclidean space. Mathematically, such a superposition is just the Fourier transform of a measure $\mu$ on $\mathbb{R}^n$,
$$
\widehat{\mu}(x):=\int e^{-2\pi ix\cdot\xi}\mathrm{d}\mu(\xi),
$$
with $x$ being a point in $n$-dimensional space.
Estimating the effects of interference, at least from the point of view of size, amounts to estimating the quantity $|\widehat{\mu}(x)|^2$ (sometimes referred to as the intensity of the wave field) in suitable normed spaces. An important example occurs when the measure $\mu$ is supported on the unit sphere $\mathbb{S}^{n-1}$, where it corresponds to plane waves of the same wavelength (monochromatic waves). This gives rise to what is referred to as the \textit{spherical extension operator}
$$
\widehat{g\mathrm{d}\sigma}(x):=\int_{\mathbb{S}^{n-1}}g(\xi)e^{-2\pi ix\cdot\xi}\mathrm{d}\sigma(\xi),
$$
taking an (integrable) input $g:\mathbb{S}^{n-1}\rightarrow\mathbb{C}$ (an amplitude) to a function on $\mathbb{R}^n$; here $\mathrm{d}\sigma$ denotes surface measure on the unit sphere $\mathbb{S}^{n-1}$ of $\mathbb{R}^n$. That this is a fundamental object of study in harmonic analysis is evident from the observation that its adjoint $$f\longmapsto\widehat{f}\Bigl|_{\mathbb{S}^{n-1}}$$ restricts the Fourier transform of a function on $\mathbb{R}^n$ to the unit sphere (a Fourier restriction operator). Evidently the study of extension operators associated with submanifolds $S$ of $\mathbb{R}^n$ (such as the sphere), referred to in harmonic analysis as \textit{Fourier restriction theory}, is of importance wherever wave-like phenomena arise, from optics and quantum mechanics to (perhaps more surprisingly) combinatorial geometry and number theory; see for example \cite{Stovall}. 

Estimating extension operators in the scale of the Lebesgue $L^p$ spaces amounts to the longstanding \textit{restriction conjecture} of Stein, which is only fully resolved when $n=2$, and interestingly identifies the curvature (such as the gaussian curvature) of $S$ as decisive in the theory; see \cite{HZ} and the references there for some recent progress in higher dimensions. However, while $L^p$ norms (with respect to some fixed measure) turn out to be effective global measures of size, they do not reveal any geometric features of the intensity $|\widehat{g\mathrm{d}\sigma}|^2$. That one might expect interesting geometric features is the basic premise of geometric optics: rays (lines) should be in clear evidence in such optical wave fields. A natural way to capture this is via general \textit{weighted} integrals of $|\widehat{g\mathrm{d}\sigma}|^2$, that is, expressions of the form
$$
\int_{\mathbb{R}^n}|\widehat{g\mathrm{d}\sigma}(x)|^2w(x)\mathrm{d}x,
$$
where $w$ is a general, usually nonnegative, locally integrable\footnote{In some of our arguments it will be convenient to assume that $w$ lies in an appropriate class of test functions, such as the Schwartz class.} function on $\mathbb{R}^n$. A good illustration of this is found in the high-frequency limiting identity
\begin{equation}\label{ah}
\limsup_{R\rightarrow\infty}R^{n-1}\int_{\mathbb{R}^{n}}|\widehat{g\mathrm{d}\sigma}(Rx)|^2w(x)\mathrm{d}x=\frac{1}{(2\pi)^{n+1}}\int_{\mathbb{S}^{n-1}}|g(\xi)|^2\left(\int_{\mathbb{R}}w(t\xi)\mathrm{d}t\right)\mathrm{d}\sigma(\xi),
\end{equation}
established (for compactly supported $w$) by Agmon and H\"ormander in \cite{AH}; see \cite{BRV}.
For finite frequencies a related problem, usually attributed to Mizohata and Takeuchi, is to determine the exponents $\alpha\geq 0$ for which the inequality
\begin{equation}\label{MTS}
\int_{B(0,R)}|\widehat{g\mathrm{d}\sigma}(x)|^2w(x)\mathrm{d}x\lesssim R^\alpha\|g\|_{L^2(\mathbb{S}^{n-1})}^2\|Xw\|_{L^\infty(\{(\omega,v):\omega\in\supp(g), v\in\langle\omega\rangle^\perp\})}
\end{equation}
holds, or possibly its weaker ``undirected" form where the $L^\infty$ norm on the right-hand side is unrestricted; see \cite{BRV}. Here 
\begin{equation}\label{ex}
Xw(\omega,v):=\int_{-\infty}^\infty w(v+t\omega)\mathrm{d}t
\end{equation}
denotes the X-ray transform, computing the integral of $w$ along the line with direction $\omega\in\mathbb{S}^{n-1}$ passing through the point $v\in\langle\omega\rangle^\perp$. We refer to Carbery, Iliopoulou and Wang \cite{CIW} for the current smallest $\alpha$ for which \eqref{MTS} is known; see also \cite{DGOWWZ, DORZ, DZ, DLWZ, Ortiz, Shayya, Wu, Mul, GWZ, BN, BNS, BGNO} for a number of recent results in weighted restriction theory, and the references there. An important feature of this statement, and most others in weighted harmonic analysis, is that the right-hand side is expressed in purely geometric (or integro-geometric) terms.
Very recently Cairo \cite{Cairo} showed that the Mizohata--Takeuchi inequality \eqref{MTS} fails in its conjectured global ($\alpha=0$) form 
\begin{equation}\label{MTSglobal}
\int_{\mathbb{R}^n}|\widehat{g\mathrm{d}\sigma}(x)|^2w(x)\mathrm{d}x\lesssim\|g\|_{L^2(\mathbb{S}^{n-1})}^2\|Xw\|_{L^\infty(\{(\omega,v):\omega\in\supp(g), v\in\langle\omega\rangle^\perp\})},
\end{equation}
even without the restriction on the X-ray norm on the right-hand side, despite being true in some notable situations (for example for radial weights \cite{BRV}, \cite{CS}, \cite{BBC3}; see also \cite{BCSV}).
That \eqref{MTSglobal} might seem plausible stems from the example of a \textit{single wavepacket}, that is, the special case where $g$ is a (modulated) characteristic function of a spherical cap. In this situation $|\widehat{g\mathrm{d}\sigma}|^2$ is easily seen to be essentially supported on a tubular region of $\mathbb{R}^n$, at which point the left-hand side of \eqref{MTSglobal} presents something akin to a line integral of the weight $w$. A further key observation is that the tube generated has axis in the direction of the normal to $\mathbb{S}^{n-1}$ at some point in the associated cap -- a feature that explains the restriction on the X-ray transform appearing on the right-hand side of \eqref{MTSglobal}.

While the Mizohata--Takeuchi conjecture in its global form \eqref{MTSglobal} is now known to be false, there are some closely related formulations that turn out to be true. One simple example is
\begin{equation}\label{MTSglobalsob}
\int_{\mathbb{R}^n}|\widehat{g\mathrm{d}\sigma}(x)|^2w(x)\mathrm{d}x\lesssim\|g\|_{L^2(\mathbb{S}^{n-1})}^2\sup_{\omega\in\supp(g)^\diamond}\|Xw(\omega,\cdot)\|_{\dot{H}^s(\langle\omega\rangle^\perp)},
\end{equation}
which holds for all values of the $L^2$-Sobolev exponent $s<(n-1)/2$. Here, for a subset $E$ of the sphere, $E^\diamond$ denotes the set of (great-circular, or geodesic) midpoints of pairs of points from $E$, and we note that the threshold $s=(n-1)/2$ is critical from the point of view of Sobolev embeddings of $\dot{H}^s(\langle\omega\rangle^\perp)$ into $L^\infty(\langle\omega\rangle^\perp)$; see \cite{BGNO}. We remark that while best-possible in this sense, the inequality \eqref{MTSglobalsob} has a right-hand side that is not entirely ``geometrically-defined" -- an important feature that one would like to retain where possible.
\subsection{Mizohata--Takeuchi inequalities and orthonormal systems}
The main purpose of this paper is to identify a new setting in which a global Mizohata--Takeuchi-type inequality (with geometric features similar to those of \eqref{MTSglobal}) may be seen to hold: that of \textit{orthonormal systems} of inputs $(g_j)$. 
The usual formalism, developed in the setting of extension operators acting on $L^2$ by Frank and Sabin \cite{FS} (see also earlier work of Frank, Lewin, Lieb and Seiringer \cite{FLLS} in the setting of Strichartz estimates for the Schr\"odinger equation with data in $L^2$), is to look for improvements to a single-function estimate by replacing the ``density" $|\widehat{g\mathrm{d}\sigma}|^2$ by the linear combination of densities $$\sum_j\lambda_j|\widehat{g_j\mathrm{d}\sigma}|^2$$ for an orthonormal sequence of inputs $(g_j)$. As a notable example, the classical Stein--Tomas inequality for the sphere,
\begin{equation}\label{stom}
\|\widehat{g\mathrm{d}\sigma}\|_{L^{\frac{2(n+1)}{n-1}}(\mathbb{R}^n)}\lesssim\|g\|_{L^2(\mathbb{S}^{n-1})},
\end{equation}
was shown to improve to
\begin{equation}\label{ONST_endpoint}
\Bigl\|\sum_j\lambda_j|\widehat{g_j\mathrm{d}\sigma}|^2\Bigr\|_{L^{\frac{n+1}{n-1}}(\mathbb{R}^n)}\lesssim \|(\lambda_j)\|_{\ell^{\frac{n+1}{n}}}.
\end{equation}
Evidently the improvement here is captured by the summability exponent $\frac{n+1}{n}>1$ appearing on the right-hand side, beating what follows trivially from the classical inequality \eqref{stom} and the triangle inequality.
In the setting of the Mizohata--Takeuchi inequalities \eqref{MTS} one might wonder whether there could be similar extensions to orthonormal systems, such as
\begin{equation}\label{MTSsys}
\int_{B(0,R)}\sum_j\lambda_j|\widehat{g_j\mathrm{d}\sigma}(x)|^2w(x)\mathrm{d}x\lesssim R^\alpha\|(\lambda_j)\|_{\ell^\beta}\|Xw\|_{L^\infty(\{(\omega,v):\omega\in K, v\in\langle\omega\rangle^\perp\})},
\end{equation}
for some $\beta>1$; here $$
K:=\bigcup_j\supp(g_j).
$$
However the inequality \eqref{MTSsys}, which  may be reformulated by $\ell^p$ duality as 
\begin{equation}\label{MTSsysdual}
\Bigl\|\left(\int_{B(0,R)}|\widehat{g_j\mathrm{d}\sigma}|^2w\right)\Bigr\|_{\ell^p}\lesssim R^{\alpha }\|Xw\|_{L^\infty(\{(\omega,v):\omega\in K, v\in\langle\omega\rangle^\perp\})},
\end{equation}
where $p=\beta'$, appears to be at odds with the usual wavepacket heuristics described earlier in the setting of \eqref{MTS}. Notice that for an orthonormal sequence of wavepackets $(g_j)$, the tubular regions that emerge are either essentially disjoint, or have distinct directions belonging to a separated set of points in the combined support set $K$, and so the left-hand side of \eqref{MTSsysdual} presents something akin to a discretisation of the $L^p$ norm of $Xw$ on $K$, rather that the $L^\infty$ norm. This suggests that for $\beta>1$ the inequality \eqref{MTSsys} is more naturally replaced with 
\begin{equation}\label{MTSsysp}
\int_{B(0,R)}\sum_j\lambda_j|\widehat{g_j\mathrm{d}\sigma}(x)|^2w(x)\mathrm{d}x\lesssim R^\alpha\|(\lambda_j)\|_{\ell^\beta}\|Xw\|_{L^p(\{(\omega,v):\omega\in K, v\in\langle\omega\rangle^\perp\})},
\end{equation} 
or equivalently
\begin{equation*}
\sum_j\left(\int_{B(0,R)}|\widehat{g_j\mathrm{d}\sigma}|^2w\right)^p\lesssim R^{\alpha p}\|Xw\|_{L^p(\{(\omega,v):\omega\in K, v\in \langle \omega\rangle^\perp\})}^p,
\end{equation*}
especially if one is considering many inputs $g_j$.
Evidently this more liberal perspective on orthonormalisation, in contrast with that in, say, \cite{FLLS} and \cite{FS}, does not give rise to inequalities that are manifestly stronger than their single-input forms. However, we note that if \eqref{MTSsysp} were to hold for one value of $\alpha$, then \eqref{MTSsys} (and hence \eqref{MTS}) would hold for a certain other since 
$$\|Xw\|_{L^p(\{(\omega,v):\omega\in K, v\in\langle\omega\rangle^\perp\})}\lesssim R^{\frac{n-1}{p}}\|Xw\|_{L^\infty(\{(\omega,v):\omega\in K, v\in\langle\omega\rangle^\perp\})}$$ 
whenever $w$ is supported in $B(0,R)$.
We further clarify these heuristics shortly and present further contextual remarks later in this introductory section.

We begin by presenting our results in the case of the extension operator for the sphere, before moving on to other submanifolds and in particular the paraboloid, where our results are naturally presented as Strichartz estimates for the Schr\"odinger equation. Our focus will be on global (scale-free and scale-invariant) estimates throughout.
\vspace{2mm}

\noindent \textbf{The sphere.}
As we have already discussed, in the context of orthonormal sequences $(g_j)$ of functions on $\mathbb{S}^{n-1}$, the analogue of the single wavepacket example described in the context of \eqref{MTSglobal} is of course an orthonormal sequence of wavepackets, where each $g_j$ is a modulated characteristic function of a cap, suitably normalised. As we reason above, this suggests that a natural analogue of \eqref{MTSglobal} for an orthonormal sequence $(g_j)$ might take the form
\begin{equation}\label{Conj:onmtsphereintrodual}
\int_{\mathbb{R}^n}\sum_j\lambda_j|\widehat{g_j\mathrm{d}\sigma}|^2w\lesssim\|(\lambda_j)\|_{\ell^{p'}}\|Xw\|_{L^p(\{(\omega,v):\omega\in K, v\in \langle \omega\rangle^\perp\})},
\end{equation}
or equivalently,
\begin{equation}\label{p}
\sum_j\left(\int_{\mathbb{R}^n}|\widehat{g_j\mathrm{d}\sigma}|^2w\right)^p\lesssim \|Xw\|_{L^p(\{(\omega,v):\omega\in K, v\in \langle \omega\rangle^\perp\})}^p
\end{equation}
for some range of $1\leq p<\infty$; we omit $p=\infty$ as \eqref{p} may then be seen to be equivalent to the false global Mizohata--Takeuchi conjecture \eqref{MTSglobal}. 
In the particular case where the sequence $(g_j)$ consists of wavepackets, their orthonormality  means that the counting measure on the left-hand side of \eqref{p} corresponds to a separated net of points in phase space (the tangent bundle of $\mathbb{S}^{n-1}$), connecting the $\ell^p$ norm on the left of \eqref{p} with the $L^p$ norm on the right. 
Turning to a general orthonormal sequence $(g_j)$, since $$\|Xw\|_{L^1(\{(\omega,v):\omega\in K, v\in \langle \omega\rangle^\perp\})}=|K|\|w\|_1,$$ the case $p=1$ of \eqref{p} reduces to the uniform pointwise estimate
\begin{equation}\label{preFS}
\sum_j|\widehat{g_j\mathrm{d}\sigma}|^2\leq |K|,
\end{equation}
which follows immediately by an application of Bessel's inequality. The focus of this work is therefore on what might be true for $p>1$, and in particular the case $p=2$.
Our main theorem in the setting of the sphere is the following:
\begin{theorem}\label{mainsphereintro}
Suppose $n\geq 3$. If $(g_j)$ is an orthonormal sequence in $L^2(\mathbb{S}^{n-1})$ then
\begin{equation}\label{onmtgenSn-1intro}
\sum_j\left(\int_{\mathbb{R}^n}|\widehat{g_j\mathrm{d}\sigma}|^2w\right)^2\leq \|Xw\|_{L^2(\{(\omega,v):\omega\in K^\diamond, v\in \langle \omega\rangle^\perp\})}^2
\end{equation}
for all signed weight functions $w$.
If $n=2$ the same is true (up to a constant factor) provided $K$ is contained in some (fixed) closed set that has no antipodal points. 
    \end{theorem}
    We recall that $K^\diamond$ is the set of all midpoints of pairs of points from $K$, suggesting that some flexibility in the restriction of the X-ray transform on the right-hand side of \eqref{p} may be appropriate when $p>1$. On the other hand, our approach permits signed weights $w$ in the statement of Theorem \ref{mainsphereintro}, and in this setting it may be seen that $K^\diamond$ cannot be replaced with $K$ in \eqref{onmtgenSn-1intro}; see Remark \ref{Remcopos}. 
    
    It is unsurprising that Theorem \ref{mainsphereintro} is more accessible than the classical Mizohata--Takeuchi inequalities \eqref{MTS} as it has a more comprehensive $L^2$ structure. In particular, the X-ray norm on the right-hand side permits some interpretation as a certain directional fractional integral norm. Specifically,
for a subset $E$ of $\mathbb{S}^{n-1}$ a routine computation reveals that
    $$\|Xw\|_{L^2(\{(\omega,v):\omega\in E, v\in \langle \omega\rangle^\perp\})}^2=\langle G_E*w, w\rangle,$$
    where $G_E(x)\sim\mathbf{1}_{\Gamma(E)}(x)/|x|^{n-1}$ and $\Gamma(E)$ is the symmetric cone in $\mathbb{R}^n$ generated by $E$.
    This is of course consistent with the well-known identity $\|Xw\|_2=c_n\|(-\Delta)^{-1/4}w\|_2$ in the unrestricted case. Such expressions have certain $L^p$ bounds by (a simple variant of) the Hardy--Littlewood--Sobolev inequality -- specifically,
$$\|Xw\|_{L^2(\{(\omega,v):\omega\in E, v\in \langle \omega\rangle^\perp\})}\lesssim |E|^{\frac{n-1}{2n(n+1)}}\|w\|_{\frac{2n}{n+1}},
$$
allowing Theorem \ref{mainsphereintro} (and our other results) to be quickly connected to the more conventional orthonormal extension estimates in the literature. 
On a related note, a benefit of permitting signed weights (as we do in Theorem \ref{mainsphereintro}) is that they allow one to deduce Sobolev estimates for $|\widehat{g\mathrm{d}\sigma}|^2$ from \eqref{onmtgenSn-1intro}. These may be seen to recover known $L^p$ orthonormal extension estimates via Sobolev embeddings; see the forthcoming contextual remarks at the end of this section.

Our proof of Theorem \ref{mainsphereintro} proceeds via the Schatten space perspective that is prevalent in the general setting of estimates for orthonormal systems. Indeed, since \eqref{p} is uniform in the collection of all orthonormal sequences $(g_j)$, thanks to \cite[Theorem 2.2]{Simon}, this recasts \eqref{p} as 
\begin{equation}\label{shp}
\|\mathcal{E}_K^*w\mathcal{E}_K\|_{\mathcal{C}^p}\lesssim\|Xw\|_{L^p(\{v\in\langle\omega\rangle^\perp, \omega\in K\})},
\end{equation}
where $\mathcal{E}_Kg:=\mathcal{E}(\mathbf{1}_Kg)$, with $\mathcal{E}g:=\widehat{g\mathrm{d}\sigma}$; this perspective is also noted in \cite{Cairo} in the setting of the original Mizohata--Takeuchi conjecture \eqref{MTSglobal}, where the Schatten norm $\|\cdot\|_{\mathcal{C}^\infty}$ coincides with the usual operator norm.
The reduction to \eqref{shp}, along with the fact that the $\mathcal{C}^2$ norm coincides with the Hilbert--Schmidt norm, allows Theorem \ref{mainsphereintro} to be further reduced to 
the assertion that $$X_0^*\mathbf{1}_{K^\diamond}-|\widehat{\mathbf{1}_K\mathrm{d}\sigma}|^2,$$
or equivalently
$$
\frac{\mathbf{1}_{K^\diamond}(\hat{x})+\mathbf{1}_{K^{\diamond}}(-\hat{x})}{|x|^{n-1}}-|\widehat{\mathbf{1}_K\mathrm{d}\sigma}(x)|^2
$$
is a positive semi-definite function (or more accurately, distribution), where $X_0f(\omega):=Xf(\omega,0)$ is the X-ray transform \eqref{ex} restricted to lines through the origin and $\hat{x}=x/|x|$.
This is established via an elementary, yet seemingly new pointwise tomographic (or ``hyperplane bundle") bound on the Fourier transform of $|\widehat{g\mathrm{d}\sigma}|^2$; see the forthcoming Lemma \ref{tomlemma} and its extension (Lemma \ref{tomlemma-generalS}) to more general hypersurfaces. 
It remains plausible that for \textit{nonnegative} weights \eqref{p} is true as stated when $p=2$ -- that is, with the X-ray transform restricted to the direction set $K$ rather than the larger $K^\diamond$. This would require one to establish the \textit{co-positive definiteness} of
$$X_0^*\mathbf{1}_{K}-c|\widehat{\mathbf{1}_K\mathrm{d}\sigma}|^2,$$
or equivalently
\begin{equation}\label{Saturday}
\frac{\mathbf{1}_{K}(\hat{x})+\mathbf{1}_{K}(-\hat{x})}{|x|^{n-1}}-c|\widehat{\mathbf{1}_K\mathrm{d}\sigma}(x)|^2
\end{equation}
for some $c>0$ independent of $K$; see Remarks \ref{boch}--\ref{impst} for clarification and further context, including how the co-positivity of \eqref{Saturday} is equivalent to the Stein-type formulation \eqref{stein} of the original Mizohata--Takeuchi problem \eqref{MTSglobal} for a special class of weights.

In addition to using a Schatten space perspective, we also develop a direct\footnote{Here we mean a direct approach to bounding the density $\sum_j \lambda_j|\widehat{g_j\mathrm{d}\sigma}|^2$, which is the way that orthonormal extension estimates are usually formulated.} approach to Theorem \ref{mainsphereintro} based on Wigner distributions, addressing an open problem raised in \cite{FS}. This appears to be related to \cite{GM}, where classical Wigner distributions are used to transfer results from kinetic theory to the study of quantum density operators; this connection is more evident in the setting of the orthonormal Strichartz estimates for the Schr\"odinger equation (where the underlying submanifold in a paraboloid) that we discuss shortly. An interesting feature of the direct (Wigner) approach is that it allows notions of orthonormality to arise organically
(where they may differ from the basic notion) and sometimes leads to stronger estimates. For example, our direct approach yields the following when $n=3$:
\begin{theorem}\label{main2intro}
Suppose that $(g_j)$ is an orthonormal sequence in $L^2(\mathbb{S}^2)$ for which $\langle g_j,\tilde{g}_k\rangle=0$ for all $j\not=k$, where $\tilde{g}:=\overline{g(-\cdot)}$. Then
\begin{equation}\label{onmtgenS2intro}
\sum_j\left(\int_{\mathbb{R}^3}|\widehat{g_j\mathrm{d}\sigma}|^2w\right)^2\leq \|Xw\|_{L^2(\{(\omega,v):\omega\in K^*, v\in \langle \omega\rangle^\perp\})}^2
\end{equation}
for all signed weight functions $w$,
where 
\begin{equation}\label{defK*}
K^*=\bigcup_j\supp(g_j)^\diamond.
\end{equation}
\end{theorem}
Evidently $K^*$ is contained in $K^\diamond$, typically with strict inclusion, and so \eqref{onmtgenS2intro} is stronger than \eqref{onmtgenSn-1intro}, although evidently under a stronger hypothesis on the sequence $(g_j)$. This stronger hypothesis appears to be rather natural given the symmetries of \eqref{onmtgenS2intro}. We refer to the forthcoming Theorem \ref{main2n} for a statement in all dimensions.

We do not present any significant results towards \eqref{p} for $p>2$. 
However, since the Schatten space approach reduces \eqref{p} to \eqref{shp} we are able to recast \eqref{p} as the co-positivity of a certain explicit $p$-tensor form when $p\in 2\mathbb{N}$; this we have already mentioned in the case $p=2$. This appears to be a novel feature of the orthonormal (as opposed to the traditional single-function) framework -- see Remark \ref{Remcopos} and Section \ref{pten}.
Furthermore, there is some evidence that \eqref{p} might hold \textit{in reverse} in the excluded range $p<1$ if the sequence $(g_j)$ is complete. We make a number of observations in support of this in Section \ref{Sect:rev}.
\begin{remark}[Interpolation]
    It is not clear to us how one might effectively interpolate within the family of estimates \eqref{p} in general, even in its weaker undirected form 
\begin{equation}\label{undirp}
\sum_j\left(\int_{\mathbb{R}^n}|\widehat{g_j\mathrm{d}\sigma}|^2w\right)^p\lesssim\|Xw\|_p^p,
\end{equation}
where the restriction of the X-ray transform to the support set $K$ is removed. We are at least able to fill the gap between the established $p=1$ and $p=2$ cases provided $\|Xw\|_p$ is replaced with the larger $\|(-\Delta)^{-\frac{1}{2p'}}w\|_p$ in \eqref{undirp}; see Section \ref{gapfill}.
\end{remark}
\vspace{2mm}

\noindent \textbf{Other submanifolds and the paraboloid.}
We now turn to submanifolds of $\mathbb{R}^n$ other than the sphere. Depending on the choice, one approach (via Schatten spaces or Wigner distributions) may be more convenient than the other. In the case of rather general hypersurfaces $S$, where \eqref{p} naturally permits a scale-invariant formulation, the Schatten approach has an advantage in that it leads to statements that are considerably easier to interpret; see the forthcoming Section \ref{Sect:gen}.
However, in the specific case of the paraboloid, the direct (Wigner) approach appears to be most effective in all respects, allowing one to treat almost orthogonal data and more effectively track supports. In this parabolic setting the Mizohata--Takeuchi-type problems are naturally formulated as Strichartz estimates for the free Schr\"odinger equation
\begin{equation}\label{schrod}
2\pi i\frac{\partial u}{\partial t}=\Delta u
\end{equation} with initial datum $u_0\in L^2(\mathbb{R}^d)$.
Here a natural form of the global (and equally false; see \cite{Cairo}) Mizohata--Takeuchi conjecture is the inequality
\begin{equation}\label{mt}\int_{\mathbb{R}^{d+1}}|u(x,t)|^2w(x,t)\mathrm{d}x\mathrm{d}t\lesssim\|u_0\|_{L^2(\mathbb{R}^d)}\|\rho^*w\|_{L^\infty(\mathbb{R}^d\times \supp(\widehat{u}_0))},\end{equation}
where 
$$
\rho^*w(x,v)=\int_{-\infty}^\infty w(x-tv,t)\mathrm{d}t.
$$
The operator $\rho^*$ is readily interpreted as a certain parametrised space-time X-ray transform, computing space-time line integrals through a point $(x,0)\in\mathbb{R}^{d+1}$ in the direction $(-v,1)$.
This global inequality, despite being false as stated (see \cite{Cairo}) is readily verifiable for \textit{single wavepackets}, which again goes some way to explain the suggested control by line integrals. As suggested in \cite{Cairo}, it remains possible that the local version
\begin{equation}\label{mte}\int_{B_R}|u(x,t)|^2w(x,t)\mathrm{d}x\mathrm{d}t\lesssim_\varepsilon R^\varepsilon\|u_0\|_{L^2(\mathbb{R}^d)}\|\rho^*w\|_{L^\infty(\mathbb{R}^d\times \supp(\widehat{u}_0))},\end{equation}
holds for each $\varepsilon>0$ (under the additional assumption that $\supp(\widehat{u}_0)$ is supported in the unit ball); here $B_R$ denotes a space-time ball of diameter $R\gg 1$.

By analogy with \eqref{p}, for \textit{orthonormal} initial data $(u_{j,0})$ it seems natural to ask if an $L^p$ version of \eqref{mt} holds, specifically:
\begin{equation}\label{onmtconj}\sum_j\left(\int_{\mathbb{R}^{d+1}} |u_j|^2w\mathrm{d}x\mathrm{d}t\right)^p\lesssim\|\rho^*w\|_{L^p(\mathbb{R}^d\times K)}^p,
\end{equation}
for some $1\leq p<\infty$, where 
$$K=\bigcup_j\supp(\widehat{u}_{j,0}).$$
Like \eqref{mt}, the inequality \eqref{onmtconj} is naturally motivated by wavepacket examples.

As in the spherical case the inequality \eqref{onmtconj} is easily seen to hold when $p=1$ by a simple application of Bessel's inequality. When $p=2$, also as in the setting of the sphere, such estimates for orthonormal systems are very accessible, and this becomes apparent via their phase-space (or Wigner) formulations. For the original global inequality \eqref{mt} this is the inequality
\begin{equation}\label{mtpsf}
\int_{\mathbb{R}^{2d}}W(u_0,u_0)(x,v)\rho^*w(x,v)\mathrm{d}x\mathrm{d}v\lesssim\|\rho^*w\|_{L^\infty(\mathbb{R}^d\times \supp(\widehat{u}_0))},
\end{equation}
where
\begin{equation}\label{eucwig}
W(u_0,u_0)(x,v):=\int_{\mathbb{R}^d}u_0\left(x+\frac{y}{2}\right)\overline{u_0\left(x-\frac{y}{2}\right)}e^{-2\pi iv\cdot y}\mathrm{d}y
\end{equation} 
is the classical Wigner distribution of $u_0$; we stress that, being equivalent to \eqref{mt}, this also fails to hold.  Evidently the subtleties in such an inequality relate to the fact that $\|W(u_0,u_0)\|_1$ is unbounded (and indeed often infinite), a fact that is closely related to the nonpositivity of Wigner distributions; see \cite{Lerner}. As we shall see, the corresponding phase-space formulation of \eqref{onmtconj} in the case $p=2$ is much more tractable as one only needs to understand Wigner distributions in $L^2$.
Here we establish the following version of \eqref{onmtconj} for $p=2$ using the direct (Wigner) approach:
\begin{theorem}\label{main} If $(u_{j,0})$ is an orthonormal system then
\begin{equation}\label{onmt} \sum_j\left(\int_{\mathbb{R}^{d+1}}|u_j|^2w\right)^2\leq\|\rho^*w\|_{L^2(M)}^2
\end{equation}
where $$M=\bigcup_j\supp(W(u_{j,0},u_{j,0}))\subseteq \bigcup_j(\supp(u_{j,0})^\diamond\times\supp(\widehat{u}_{j,0})^\diamond),$$ and $E^\diamond=\frac{1}{2}(E+E)$ is the midpoint set of a subset $E\subseteq\mathbb{R}^d$. 
\end{theorem}
As we shall see in Section \ref{Sect:par}, the direct approach that we take to Theorem \ref{main} is particularly simple as the appropriate Wigner-type distribution is the classical one.
We note in passing that \eqref{onmt} is Fourier-invariant in the sense that it is unchanged if the initial data $(u_{j,0})$ is replaced with $(\widehat{u}_{j,0})$ -- a routine calculation using pseudo-conformal symmetry. Indeed one might consider strengthening the suggested inequality \eqref{onmtconj} to a similarly invariant statement for any $p$. It is also appropriate to mention \cite{Jam,Jan} here, where it is shown that the support of the Wigner distribution of a non-zero $L^2$ function necessarily has infinite Lebesgue measure.
If the sequence $(u_j)$ consists of a single solution then \eqref{onmt} may essentially be found in Dendrinos, Mustata and Vitturi \cite{DMV}; see \cite{BGNO} for further discussion and developments inspired by \eqref{mt} based on Wigner-type distributions. 

As in the setting of \eqref{p} for the sphere, we do not provide any significant evidence in support of \eqref{onmtconj} for any $p>2$. While both are suggested quite tentatively for $p>2$, they do interact naturally with a range of related inequalities in the literature. We conclude this section with some discussion of this type.
\subsection{Contextual remarks}
A seemingly natural case of the suggested inequality \eqref{p} is when $p=n+1$ since it allows one to invoke the endpoint X-ray estimate 
\begin{equation}\label{chr}
\|Xw\|_{n+1}\lesssim\|w\|_{\frac{n+1}{2}}
\end{equation}
of Christ \cite{Christk}, upon which \eqref{p} (in the equivalent form \eqref{Conj:onmtsphereintrodual}) is seen to imply the endpoint orthonormal Stein--Tomas inequality \eqref{ONST_endpoint}
of Frank and Sabin \cite{FS}. A similar remark may be made in the more traditional single-input setting 
of \eqref{p}, namely
\begin{equation}\label{single}
\int_{\mathbb{R}^n}|\widehat{g\mathrm{d}\sigma}|^2w\lesssim\|Xw\|_{L^p(\{(\omega,v):\omega\in\supp(g), v\in\langle\omega\rangle^\perp\})}\|g\|_2^2.
\end{equation}
In addition to implying the classical Stein--Tomas restriction theorem via \eqref{chr}, this also implies \eqref{MTS} with $\alpha=\frac{n-1}{p}$ via the trivial bound $$\|X(w\mathbf{1}_{B(0,R)})\|_p\lesssim R^{\frac{n-1}{p}}\|Xw\|_\infty.$$
In light of these things it may be interesting to compare the suggested inequality \eqref{single} in the particular case $p=n+1$ with the results of Carbery, Iliopoulou and Wang in \cite{CIW}. We remark that the mixed norm variant  
$$
\int_{B(0,R)}|\widehat{g\mathrm{d}\sigma}|^2w\leq C_\varepsilon R^\varepsilon \sup_{\omega\in\supp(w)}\|Xw(\omega,\cdot)\|_{L^{\frac{n+1}{2}}(\langle\omega\rangle^\perp)}\|g\|_2^2,
$$
of \eqref{single} follows quickly from Theorem 4.1 of \cite{CIW}. Evidently if \eqref{MTS} were to fail for some $\alpha>0$ then \eqref{single} (and thus \eqref{p}) would fail for $p\geq\frac{n-1}{\alpha}$; similar conclusions follow for submanifolds other than the sphere, for which we refer forward to Remark \ref{Rem:CZ}.

The orthonormal Stein--Tomas inequality \eqref{ONST_endpoint} is optimal in the sense that it is not possible to raise the summability exponent $\frac{n+1}{n}$ on the right-hand side (see \cite{FS}). Of course, by simply applying the triangle inequality on the left-hand side and then the classical (single-function) Stein--Tomas inequality, we obtain 
\begin{equation*} 
\bigg\| \sum_j\lambda_j|\widehat{g_j\mathrm{d}\sigma}|^2 \bigg\|_{L^\frac{n+1}{n-1}(\mathbb{R}^n)} \lesssim \|(\lambda_j)\|_{\ell^{1}}.
\end{equation*}
This argument makes no use of the orthogonality of the $(g_j)$ and one can view the size of the summability exponent on the right-hand side as a means of quantifying the gain from the orthogonality. More generally, Frank--Sabin \cite{FS} proved
\begin{equation} \label{ONST}
\bigg\| \sum_j\lambda_j|\widehat{g_j\mathrm{d}\sigma}|^2 \bigg\|_{L^q(\mathbb{R}^n)} \lesssim \|(\lambda_j)\|_{\ell^{\frac{n-1}{n}q}}
\end{equation}
for all orthonormal sequences $(g_j)$ in $L^2(\mathbb{S}^{n-1})$ if $q \in [\frac{n+1}{n-1},\infty]$, and the summability exponent on the right-hand side is optimal for all such $q$. We remark that \eqref{ONST} with $q = \infty$ follows by a straightforward argument using Bessel's inequality.

We have already observed that \eqref{p} with $p = n + 1$ yields \eqref{ONST} in the endpoint case $q = \frac{n+1}{n-1}$. Unsurprisingly, smaller values of $p$ also make contact with \eqref{ONST}. By $\ell^p$ duality and invoking the family of X-ray estimates $\|Xw\|_p \lesssim \|w\|_{q'}$ for $q' \in [1,\frac{n+1}{2}]$ and $\frac{1}{p'} = \frac{n}{n-1}\frac{1}{q}$, we may deduce \eqref{ONST} for any $q \in [\frac{n+1}{n-1},\infty]$. In particular, Theorem \ref{mainsphereintro} implies \eqref{ONST} when $q = \frac{2n}{n-1}$ (and hence $q \in [\frac{2n}{n-1},\infty]$ by interpolation). In fact, from Theorem \ref{mainsphereintro} we actually obtain the following refined smoothing estimate
\begin{equation} \label{STsmoothing}
\bigg\| \sum_j\lambda_j|\widehat{g_j\mathrm{d}\sigma}|^2 \bigg\|_{\dot{H}^{1/2}(\mathbb{R}^n)} \lesssim \|(\lambda_j)\|_{\ell^2}
\end{equation}
by duality and the fact that $\|Xw\|_2 = c_n\|w\|_{\dot{H}^{-1/2}(\mathbb{R}^n)}$. Here, $\dot{H}^{s}(\mathbb{R}^n)$ denotes a homogeneous Sobolev space and we immediately deduce \eqref{ONST} when $q = \frac{2n}{n-1}$ by a fractional Sobolev estimate.

In the spirit of the above discussion on the Stein--Tomas inequality, for any functional inequality whose input functions belong to a Hilbert space, it is a natural question to consider extensions to orthonormal sequences. A pioneering example of this is the Lieb--Thirring inequality (see \cite{LT}), which may be stated as
\begin{equation} \label{LiebThirring}
   \bigg\| \sum_j |g_j|^2 \bigg\|_{L^{\frac{n+2}{n}}(\mathbb{R}^n)} \lesssim \bigg( \sum_j \|g_j\|_{\dot{H}^1(\mathbb{R}^n)}^2 \bigg)^{\frac{n}{n+2}}
\end{equation}
for orthonormal sequences $(g_j)$ in $L^2(\mathbb{R}^n)$. This is an extension of a Gagliardo--Nirenberg--Sobolev inequality to orthonormal sequences and played a key role in the proof of stability of matter by Lieb and Thirring in \cite{LT}. Ultimately, stability of matter is concerned with lower bounds on certain ground state energies and was first established by Dyson and Lenard \cite{DL1, DL2}. The argument by Lieb and Thirring provided a shorter proof and gave sharper lower bounds and, as a result, the optimal constant in \eqref{LiebThirring} has become an object of significant interest. We refer the interested reader to \cite{Lieb_BAMS} and \cite{Frank_survey} for further details on Lieb--Thirring inequalities.

Pertinent to the present paper are the orthonormal Strichartz estimates for the free Schr\"odinger propagator first studied by Frank--Lewin--Lieb--Seiringer \cite{FLLS} and subsequent developments by Chen--Hong--Pavlovic \cite{CHP} and Frank--Sabin \cite{FS}. These estimates are an important tool for rigorously understanding the dynamics of large numbers of fermions (see, for example, \cite{CHP, CHP2, FS, LS1, LS2}). A model for a system of $N$ fermions is provided by the coupled Hartree system
\[
i\partial_t u_j = (-\Delta + V * \rho)u_j, \qquad u_j(0) = u_{j,0} 
\]
for $j = 1,\ldots,N$. Here, $\rho(x,t) = \sum_{k = 1}^N |u_k(x,t)|^2$ represents the total density of particles, and the function $V$ is an interaction potential. The Pauli exclusion principle states that two fermions cannot simultaneously occupy the same quantum state and from this it is natural to impose an orthonormality condition on the initial data $(u_{j,0})$. Orthonormal estimates of the type we consider in this paper may also be of interest in optics, and specifically the theory of structured light; see \cite{FOD}.

When $d = 1$, Theorem \ref{main} has connections to the orthonormal Strichartz estimates in \cite{CHP, FLLS, FS}. To see this, note first that from Plancherel's theorem we may write
\begin{equation} \label{1Dspecial}
\|\rho^*w\|_{L^2(\mathbb{R} \times \mathbb{R})} = c\|w\|_{L^{2}_{t}\dot{H}_x^{-1/2}(\mathbb{R} \times \mathbb{R})}
\end{equation}
and therefore, by Theorem \ref{main} and duality, we obtain
$$
\bigg\| \sum_j\lambda_j |u_j|^2\bigg\|_{L^2_t\dot{H}^{1/2}_x(\mathbb{R} \times \mathbb{R})}\lesssim \|(\lambda_j)\|_{\ell^2}.
$$
This recovers a one-dimensional estimate established in \cite[Theorem 3.3]{CHP}. Furthermore, by a critical Sobolev estimate
we deduce
\begin{equation} \label{BMO}
\bigg\|\sum_j\lambda_j |u_j|^2\bigg\|_{L^2_{t}\mathrm{BMO}_x(\mathbb{R} \times \mathbb{R})}\lesssim \|(\lambda_j)\|_{\ell^2}.
\end{equation}
The stronger inequality with $\mathrm{BMO}_x$ replaced by $L^\infty_x$ is known to fail, and Frank--Sabin \cite{FSsurvey} raised the question of whether the $L^\infty_x$-version holds if we replace the right-hand side of \eqref{BMO} by $\|(\lambda_j)\|_{\ell^{\beta}}$ for some $\beta \in (1,2)$ (see \cite{BLN} for some progress on this).

 Simply invoking the triangle inequality and the conservation of energy property satisfied by the free Schr\"odinger propagator, we obtain
\begin{equation*}
\bigg\|\sum_j\lambda_j |u_j|^2\bigg\|_{L^\infty_{t}L^1_x(\mathbb{R} \times \mathbb{R})}\lesssim \|(\lambda_j)\|_{\ell^1}.
\end{equation*}
Complex interpolation with \eqref{BMO} yields
\begin{equation} \label{onStrichartz1D}
\bigg\|\sum_j\lambda_j |u_j|^2\Bigr\|_{L^{q/2}_{t}L^{r/2}_x(\mathbb{R} \times \mathbb{R})}\lesssim \|(\lambda_j)\|_{\ell^\beta}
\end{equation}
whenever 
\[
0 < \frac{1}{r} \leq \frac{1}{2}, \quad \frac{1}{q} = \frac{1}{2}\bigg(\frac{1}{2} - \frac{1}{r}\bigg), \quad \frac{1}{\beta} = \frac{1}{2} + \frac{1}{r}
\]
and this recovers \cite[Theorem 8]{FS} when $d=1$. Recalling that our proof of Theorem \ref{main} is based on the Wigner transform rather than a Schatten space perspective, as far as we are aware our approach provides the first direct proof of the orthonormal Strichartz estimates in \eqref{onStrichartz1D}, thus addressing an open problem raised by Frank--Lewin--Lieb--Seiringer \cite{FLLS}.

For $d \geq 2$ it is less apparent to us whether Theorem \ref{main} has connections to existing orthonormal Strichartz estimates.
Observe that the \textit{global} $L^2$ norm of $\rho^*w$ is often infinite for $d \geq 2$, highlighting the importance of the restriction to the set $M$ in Theorem \ref{main}. Nevertheless, the proposed inequality \eqref{onmtconj} for a fixed $p > d + 1$ would imply the  orthonormal Strichartz estimate
\begin{equation} \label{onStrichartz}
\bigg\|\sum_j\lambda_j |u_j|^2\bigg\|_{L^{q/2}_{t}L^{r/2}_x(\mathbb{R} \times \mathbb{R}^d)}\lesssim \|(\lambda_j)\|_{\ell^\beta}
\end{equation}
with $\beta = p'$, $\frac{q}{2} = \frac{p}{d}$ and $\frac{r}{2} = \frac{p}{p-2}$. For this we use the estimate
\begin{equation} \label{rhostarbound}
\| \rho^* w \|_{L^{\beta'}_{x,v}(\mathbb{R}^d \times \mathbb{R}^d)} \lesssim \|w\|_{L^{(q/2)'}_tL^{(r/2)'}_x}
\end{equation}
which is known to hold when 
\begin{equation} \label{qrbetarange}
    \frac{d-1}{2(d+1)} < \frac{1}{r} \leq \frac{1}{2}, \quad \frac{1}{q} = \frac{d}{2}\bigg(\frac{1}{2} - \frac{1}{r}\bigg), \quad \frac{1}{\beta} = \frac{1}{2} + \frac{1}{r}.
\end{equation}
In fact, by duality the X-ray estimate \eqref{rhostarbound} is equivalent to the Strichartz estimate for the kinetic transport equation
\begin{equation} \label{kineticStrichartz}
\| \rho f \|_{L^{q/2}_tL^{r/2}_x} \lesssim \|f\|_{L^{\beta}_{x,v}},
\end{equation}
where $\rho$ is the velocity averaging operator given by
\[
\rho f(x,t) = \int_{\mathbb{R}^d} f(x+tv,v) \mathrm{d}v.
\]
The family of estimates \eqref{kineticStrichartz} go back to work of Castella--Perthame \cite{CP}, were extended up to but not including the endpoint $r = \frac{2(d+1)}{d-1}$ by Keel--Tao \cite{KT}, and the failure of the endpoint was shown in \cite{GP} and \cite{BBGL} when $d=1$ and $d \geq 2$, respectively. In particular, the above argument shows that if \eqref{onmtconj} is true for all $p > d + 1$, then the orthonormal Strichartz estimates \eqref{onStrichartz} follow under the conditions in \eqref{qrbetarange}. This means we would be able to derive all of the orthonormal Strichartz estimates obtained by Frank--Sabin in \cite[Theorem 8]{FS} when $d \geq 2$; for \eqref{onStrichartz}, it is known the summability exponent $\beta$ given by \eqref{qrbetarange} is optimal for such $q$ and $r$ (see \cite{FLLS}).

Interestingly, Sabin \cite{Sabin} showed that \eqref{kineticStrichartz} follows from \eqref{onStrichartz} under the conditions in \eqref{qrbetarange} by a semiclassical limiting argument. The above discussion raises the possibility that the converse is also true and tentatively suggests that \eqref{onmtconj} is one pathway to this.
\begin{remark}[Failure for $p>n+1$ for many compact convex $C^2$ hypersurfaces]\label{Rem:CZ}
    Since earlier versions of this paper, Cairo and Zhang \cite{CZ} have shown that the local Mizohata--Takeuchi inequality 
$$
\int_{B(0,R)}|\widehat{g\mathrm{d}\sigma}(x)|^2w(x)\mathrm{d}x\lesssim R^\alpha \|Xw\|_\infty\|g\|_{L^2(S)}
$$
fails whenever $\alpha<\frac{n-1}{n+1}$, the threshold identified earlier by Carbery, Iliopoulou and Wang \cite{CIW}, for many compact convex $C^2$ hypersurfaces $S$; here $\mathrm{d}\sigma$ denotes surface measure on $S$. Since $$\|Xw\|_{L^p(B(0,R)}\lesssim R^{\frac{n-1}{p}}\|Xw\|_\infty,$$ it follows that 
\begin{equation}\label{mtp}
\int_{B(0,R)}|\widehat{g\mathrm{d}\sigma}(x)|^2w(x)\mathrm{d}x\lesssim \|Xw\|_p\|g\|_{L^2(S)}
\end{equation}
and its natural extensions to orthonormal systems (see Section \ref{Sect:gen})
fail for such hypersurfaces whenever $p>n+1$. This provides a further indication that the exponent $p=n+1$ may be of particular interest; see \eqref{chr}. It remains plausible that \eqref{mtp} holds for the sphere or piece of paraboloid for any $p<\infty$, and further that \eqref{p} and \eqref{onmtconj} hold in some form for all $p<\infty$.
\end{remark}
\vspace{2mm}
\noindent\textbf{The structure of the paper.} In Section \ref{Sect:sp} we present our arguments in the setting of the sphere, before moving to more general manifolds $S$ in Section \ref{Sect:gen}. In Section \ref{Sect:par} we treat the case of the paraboloid, presenting our results as weighted Strichartz estimates for the Schr\"odinger equation. Section \ref{Sect:remarks} consists of a number of observations concerning \eqref{p} and its variants when $p\not=2$: Section \ref{gapfill} establishes some natural interpolants of the $p=1$ and $p=2$ cases,
Section \ref{pten} observes that \eqref{p} (when $p$ is an even integer) may be recast as the co-positivity of certain $p$-tensor forms, and Section \ref{Sect:rev} provides some evidence that \eqref{p} might hold in reverse when $p\leq 1$ under a completeness assumption on the orthonormal sequence.

\vspace{2mm}
\noindent\textbf{Acknowledgments.} 
The first and fifth authors are supported by EPSRC Grant EP/W032880/1. The second author is supported by JSPS Kakenhi grant numbers 22H00098, 23K25777, 24H00024. We thank Tony Carbery, Mark Dennis, Michal Kocvara and Amy Tierney for helpful conversations and correspondence.

\section{The sphere}\label{Sect:sp}
As we have indicated in the introduction, we have two approaches to establishing our orthonormal extension estimates. The first is via a Schatten space perspective that is prevalent in such settings (used for Theorem \ref{mainsphereintro}), and the second is via phase-space methods involving a spherical form of the Wigner distribution (used for Theorem \ref{main2intro} and its forthcoming extension to higher dimensions).

\subsection{The Schatten space approach} \label{subsection:Schatten}
In this section we prove Theorem \ref{mainsphereintro} by establishing the bound
\begin{equation} \label{e:Thm2-Support}
\| \mathcal{E}_K^* w \mathcal{E}_K \|_{\mathcal{C}^2}^2 \leq \|Xw\|_{L^2(\{(\omega,v):\omega\in K^\diamond, v\in \langle \omega\rangle^\perp\})}^2,
\end{equation}
where $\mathcal{E}_Kg :=\mathcal{E}(\mathbf{1}_Kg)$. 
Recall that this is equivalent to \eqref{onmtgenSn-1intro} since the supremum of the left-hand side of \eqref{onmtgenSn-1intro} over all orthonormal sequences $(g_j)$ coincides with the left-hand side of \eqref{e:Thm2-Support}.

For the left-hand side of \eqref{e:Thm2-Support} a straightforward computation reveals that
\[
\mathcal{E}_K^* w \mathcal{E}_Kg(\omega) = \int_{\mathbb{S}^{n-1}} g(\theta) \mathbf{1}_K(\omega)\mathbf{1}_K(\theta) \widehat{w}(\theta - \omega) \mathrm{d}\sigma(\theta),
\]
and so
\begin{align*}
    \| \mathcal{E}_K^* w \mathcal{E}_K \|_{\mathcal{C}^2}^2 & = \int_{\mathbb{S}^{n-1}} \int_{\mathbb{S}^{n-1}}  |\widehat{w}(\theta - \omega)|^2 \mathrm{d}\sigma_K(\omega)\mathrm{d}\sigma_K(\theta)
    \\
    &=\int_{\mathbb{R}^n}|\widehat{w}(\xi)|^2\mathrm{d}\sigma_K*\widetilde{\mathrm{d}\sigma}_K(\xi)\mathrm{d}\xi,
\end{align*}
where $\mathrm{d}\sigma_K := \mathbf{1}_K \mathrm{d}\sigma$. For the right-hand side of \eqref{e:Thm2-Support} we use the classical X-ray formula $\mathcal{F}_v Xw(\omega,\xi)=\widehat{w}(\xi)$ for $\xi\in\langle\omega\rangle^\perp$, followed by Plancherel's theorem on $\langle\omega\rangle^\perp$, to write
\begin{eqnarray*}
    \begin{aligned}
\|Xw\|_{L^2(\{(\omega,v):\omega\in K^\diamond, v\in\langle\omega\rangle^\perp\})}^2&=\int_{K^\diamond}\int_{\langle\omega\rangle^\perp}|\widehat{w}(\xi)|^2\mathrm{d}\xi\mathrm{d}\sigma(\omega)\\&=\int_{K^\diamond}\int_{\mathbb{R}^n}|\widehat{w}(\xi)|^2\delta(\omega\cdot\xi)\mathrm{d}\xi\mathrm{d}\sigma(\omega)\\&=\int_{\mathbb{R}^n}|\widehat{w}(\xi)|^2\left(\int_{K^\diamond}\delta(\omega\cdot\xi)\mathrm{d}\sigma(\omega)\right)\mathrm{d}\xi.
    \end{aligned}
\end{eqnarray*}
Thus \eqref{e:Thm2-Support} is equivalent to the inequality
\begin{equation}\label{essence}
\mathrm{d}\sigma_K*\widetilde{\mathrm{d}\sigma}_K(\xi)\leq \int_{K^\diamond}\delta(\omega\cdot\xi)\mathrm{d}\sigma(\omega),
\end{equation}
or in other words,
$$
\mathrm{d}\sigma_K*\widetilde{\mathrm{d}\sigma}_K\leq \mathcal{R}_0^*(\mathbf{1}_{K^\diamond})
$$
where
\begin{equation}\label{resrad}
\mathcal{R}_0f(\omega):=\int_{\mathbb{R}^n}\delta(\omega\cdot\xi)f(\xi)\mathrm{d}\xi
\end{equation}
is the Radon transform on $\mathbb{R}^n$ restricted to hyperplanes passing through the origin.
This we establish via the following lemma applied with $g=\mathbf{1}_K$.
    \begin{lemma}[Tomographic estimate]\label{tomlemma} Suppose $n\geq 3$.
    For a suitable $g:\mathbb{S}^{n-1}\rightarrow\mathbb{R}_+$ define $g^\diamond:\mathbb{S}^{n-1}\rightarrow\mathbb{R}_+$ by the ``sup-autocorrelation" formula
    $$
    g^\diamond(\omega)=\sup_{\omega'}g(\omega')g(\omega''),
    $$
    where $\omega''$ is such that $\omega$ is a great-circular (or geodesic) midpoint of $\omega'$ and $\omega''$. Then 
    \begin{equation}\label{functionaltom}
    g\mathrm{d}\sigma*\widetilde{g\mathrm{d}\sigma}\leq \mathcal{R}_0^*g^\diamond.
    \end{equation} 
    If $n=2$ then a similar conclusion (with a different implicit constant) holds provided $\supp(g)$ is contained in some (fixed) closed set that contains no antipodal points.
    \end{lemma}
    \begin{proof}
      It suffices to prove that
      \begin{equation}\label{weak}
      \int_{\mathbb{S}^{n-1}}\int_{\mathbb{S}^{n-1}}\varphi(\omega'-\omega'')g(\omega')g(\omega'')\mathrm{d}\sigma(\omega'')\mathrm{d}\sigma(\omega')\leq \int_{\mathbb{S}^{n-1}}g^\diamond(\omega)\mathcal{R}_0\varphi(\omega)\mathrm{d}\sigma(\omega)
      \end{equation}
      for all nonnegative test functions $\varphi$ on $\mathbb{R}^n$. 
      For fixed $\omega'$ we make the change of variables $\omega''=R_\omega\omega'$ where
      \begin{equation}\label{omegadef}
R_\omega\omega'=2(\omega\cdot\omega')\omega-\omega',
\end{equation}
noting that the new variable $\omega$ is a great-circular (or geodesic) midpoint of $\omega'$ and $\omega''$. Evidently $\omega\mapsto R_\omega\omega'$ is not injective on $\mathbb{S}^{n-1}$ since $R_{-\omega}\omega'=R_\omega\omega'$. However, since it maps each component of $\mathbb{S}^{n-1}\backslash\langle\omega'\rangle^\perp$ bijectively to $\mathbb{S}^{n-1}\backslash\{-\omega'\}$ with $\mathrm{d}\sigma(\omega'')=2^{n-1}|\omega\cdot\omega'|^{n-2}\mathrm{d}\sigma(\omega)$ (see \cite{BGNO}), it follows that
      \begin{eqnarray*}
          \begin{aligned}
      \int_{\mathbb{S}^{n-1}}\int_{\mathbb{S}^{n-1}}\varphi(\omega'-\omega'')&g(\omega'')g(\omega')\mathrm{d}\sigma(\omega'')\mathrm{d}\sigma(\omega')\\
      &=2^{n-2}\int_{\mathbb{S}^{n-1}}\int_{\mathbb{S}^{n-1}}\varphi(\omega'-R_\omega\omega')g(\omega')g(R_\omega\omega')|\omega\cdot\omega'|^{n-2}\mathrm{d}\sigma(\omega)\mathrm{d}\sigma(\omega').
      \end{aligned}
      \end{eqnarray*}
      Next, for fixed $\omega$, we make the change of variables $\xi=\omega'-R_\omega\omega'$. This is a two-to-one map from $\mathbb{S}^{n-1}\backslash\langle\omega\rangle^\perp$ to the ball $B(0,2)$ in $\langle\omega\rangle^\perp$, which becomes one-to-one on each component hemisphere $\mathbb{S}^{n-1}_{\pm}$. Using the formula $\mathrm{d}\xi=2^{n-1}|\omega\cdot\omega'|\mathrm{d}\sigma(\omega')$ we have
      \begin{eqnarray*}
          \begin{aligned}
      2^{n-2}\int_{\mathbb{S}^{n-1}}\varphi(\omega'-R_\omega\omega')&g(\omega')g(R_\omega\omega')|\omega\cdot\omega'|^{n-2}\mathrm{d}\sigma(\omega')\\
      &=2^{n-2}\int_{\mathbb{S}^{n-1}_+}\varphi(\omega'-R_\omega\omega')g(\omega')g(R_\omega\omega')|\omega\cdot\omega'|^{n-2}\mathrm{d}\sigma(\omega')\\&+2^{n-2}\int_{\mathbb{S}^{n-1}_-}\varphi(\omega'-R_\omega\omega')g(\omega')g(R_\omega\omega')|\omega\cdot\omega'|^{n-2}\mathrm{d}\sigma(\omega')\\
      &=\frac{1}{2}\int_{\langle\omega\rangle^\perp\cap B(0,2)}\varphi(\xi)g(\omega'(\xi))g(R_\omega\omega'(\xi))|\omega\cdot\omega'(\xi)|^{n-3}\mathrm{d}\xi\\
      &+\frac{1}{2}\int_{\langle\omega\rangle^\perp\cap B(0,2)}\varphi(\xi)g(-\omega'(\xi))g(-R_\omega\omega'(\xi))|\omega\cdot\omega'(\xi)|^{n-3}\mathrm{d}\xi\\
      &\leq g^\diamond(\omega)\mathcal{R}_0\varphi(\omega)
      \end{aligned}
      \end{eqnarray*}
      whenever $n\geq 3$.
      The inequality \eqref{weak} follows on integrating in $\omega$.
      Here we have used that $g^\diamond$ is an even function.
      Note that for $n=2$ the additional hypothesis means that the singularity from the jacobian factor is removed. 
    \end{proof}
    We conclude this subsection with a number of remarks pertaining to Theorem \ref{mainsphereintro} and the above proof, along with the prospect that \eqref{p} might hold for $p=2$ as stated.
    \begin{remark}\label{boch}
    The inequality \eqref{functionaltom} amounts to the statement that
    \begin{equation}\label{bnobs}
    |\widehat{g\mathrm{d}\sigma}|^2\leq_{pd}X_0^*g^\diamond
    \end{equation}
    meaning that the function, or more accurately, distribution
    \begin{equation}\label{ftpd}   X_0^*g^\diamond-|\widehat{g\mathrm{d}\sigma}|^2
    \end{equation}
    is positive semi-definite in the sense that
    $$
    \int_{\mathbb{R}^n}(X_0^*g^\diamond-|\widehat{g\mathrm{d}\sigma}|^2)(x)w*\tilde{w}(x)\mathrm{d}x=\int_{\mathbb{R}^n\times\mathbb{R}^n}(X_0^*g^\diamond-|\widehat{g\mathrm{d}\sigma}|^2)(x-y)w(x)w(y)\mathrm{d}x\mathrm{d}y\geq 0
    $$
    for all real-valued test functions $w$. We recall that $$X_0f(\omega):=\int_{\mathbb{R}}f(t\omega)\mathrm{d}t$$
    is the X-ray transform restricted to lines through the origin in $\mathbb{R}^n$ and note the elementary formula \begin{equation}\label{formula}
    X_0^*f(x)=|x|^{-(n-1)}(f(\hat{x})+f(-\hat{x})).
    \end{equation} The point is that $X_0^*f$ is the Fourier transform of $\mathcal{R}_0^*f$,
    so that
    $$
    \int_{\mathbb{R}^n}(X_0^*g^\diamond-|\widehat{g\mathrm{d}\sigma}|^2)(x)w*\tilde{w}(x)\mathrm{d}x=\int_{\mathbb{R}^n}(\mathcal{R}_0^*g^\diamond(\xi)-g\mathrm{d}\sigma*\widetilde{g\mathrm{d}\sigma}(\xi))|\widehat{w}(\xi)|^2\mathrm{d}\xi\geq 0.
    $$
    \end{remark}
\begin{remark}[Non-negative weights and co-positivity]\label{Remcopos}
Our proof of Theorem \ref{mainsphereintro} reveals that $K^\diamond$ cannot be replaced with $K$ as suggested by \eqref{p}, if \textit{signed} weights $w$ are to be permitted. In the context of \textit{nonnegative} weights it remains plausible that it can since \eqref{e:Thm2-Support} and \eqref{essence} cease to be manifestly equivalent. That one might be able to do this is the assertion that 
$$X_0^*\mathbf{1}_K-|\widehat{\mathbf{1}_K\mathrm{d}\sigma}|^2$$ is a \textit{co-positive semi-definite} function, or if one wishes to allow a constant factor, the assertion that
\begin{equation}\label{copo}
X_0^*\mathbf{1}_K-c|\widehat{\mathbf{1}_K\mathrm{d}\sigma}|^2
\end{equation}
is a co-positive definite function for some $c>0$ independent of $K$. This is captured by the identity
\begin{eqnarray*}
    \begin{aligned}
\|Xw\|_{L^2(\{(\omega,v):\omega\in K, v\in\langle\omega\rangle^\perp\})}^2
-
c\| \mathcal{E}_K^* w \mathcal{E}_K \|_{\mathcal{C}^2}^2&=
\int_{\mathbb{R}^n\times\mathbb{R}^n}(X_0^*\mathbf{1}_K-c|\widehat{\mathbf{1}_K\mathrm{d}\sigma}|^2)(x-y)w(x)w(y)\mathrm{d}x\mathrm{d}y\\&=\int_{\mathbb{R}^n}(X_0^*\mathbf{1}_K-c|\widehat{\mathbf{1}_K\mathrm{d}\sigma}|^2)(x)w*\tilde{w}(x)\mathrm{d}x,
\end{aligned}
\end{eqnarray*}
clarifying what we mean by the co-positivite definiteness of \eqref{copo}, that is, positivity of the associated quadratic form on the cone of non-negative test functions $w$. By \eqref{formula}, the co-positive definiteness of \eqref{copo} becomes the (uniform in $K$) relation
\begin{equation}\label{bigquestion}
|\widehat{\mathbf{1}_K\mathrm{d}\sigma}(x)|^2\lesssim_{cpd}\frac{\mathbf{1}_{K}(\hat{x})+\mathbf{1}_{K}(-\hat{x})}{|x|^{n-1}},
\end{equation} 
which is something that follows quickly by stationary phase in the very special case that $K=\mathbb{S}^{n-1}$ since it is then true as a pointwise inequality. Here the relation $\lesssim_{cpd}$ is defined analogously to $\lesssim_{pd}$.
The suggested relation \eqref{bigquestion} should be compared with the relation $$|\widehat{\mathbf{1}_K\mathrm{d}\sigma}(x)|^2\leq_{pd}\frac{\mathbf{1}_{K^\diamond}(\hat{x})+\mathbf{1}_{K^\diamond}(-\hat{x})}{|x|^{n-1}}$$ that follows from Lemma \ref{tomlemma}.  If true \eqref{bigquestion} would seem rather subtle as the relation
$$
\frac{\mathbf{1}_{K^\diamond}(\hat{x})+\mathbf{1}_{K^\diamond}(-\hat{x})}{|x|^{n-1}}\lesssim_{cpd}\frac{\mathbf{1}_{K}(\hat{x})+\mathbf{1}_{K}(-\hat{x})}{|x|^{n-1}}
$$
uniformly in $K$ is quickly seen to be false (when $n=2$ in the first instance) by taking $K$ to be an $N$th generation Cantor middle-thirds construction on $\mathbb{S}^{n-1}$; the point being that $K$ may have arbitrarily small measure whilst $|K^\diamond|\sim 1$. 
It seems appropriate to note that testing a matrix for co-positive definiteness is known to be NP-hard \cite{MK}. 
\end{remark}
\begin{remark}[Functional forms]\label{ffp}
    Given the established functional relation \eqref{bnobs}, and the limiting identity \eqref{ah}, one might expect the suggested relation \eqref{bigquestion} to also have a functional generalisation taking the form
    \begin{equation}\label{bigquestionfun}
        |\widehat{g\mathrm{d}\sigma}(x)|^2\lesssim_{cpd}X_0^*(|g|^2)(x)=\frac{|g(\hat{x})|^2+|g(-\hat{x})|^2}{|x|^{n-1}}
    \end{equation}
    with implicit constant independent of $g$. 
If true, \eqref{bigquestionfun} should be contrasted with the pointwise ``far-field" relation
\begin{equation}
    \widehat{g\mathrm{d}\sigma}(x)=\frac{e^{-2\pi i|x|}g(\hat{x})}{(-i)^{\frac{n-1}{2}}|x|^{\frac{n-1}{2}}}+\frac{e^{2\pi i|x|}g(-\hat{x})}{i^{\frac{n-1}{2}}|x|^{\frac{n-1}{2}}}+ O\left(\frac{1}{|x|^{\frac{n+1}{2}}}\right)
\end{equation}
that may be established for smooth functions $g$ by stationary phase as $|x|\rightarrow\infty$. In particular it follows that
\begin{equation}\label{bigquestionfunstat}
        |\widehat{g\mathrm{d}\sigma}(x)|^2\lesssim\frac{|g(\hat{x})|^2+|g(-\hat{x})|^2}{|x|^{n-1}}+O\left(\frac{1}{|x|^{n+1}}\right).
    \end{equation}
However, the implicit constant in the $O$ notation here depends on the regularity of $g$, and is therefore far from uniform. From a physical perspective \eqref{bigquestionfun} asserts a certain uniform control of the intensity of a wavefield by that of its far field; we refer to \cite{CK} for further context of this type.
\end{remark}
\begin{remark}[Connection to Stein's conjecture for the extension operator]\label{impst}
The suggested functional relation \eqref{bigquestionfun} is easily seen to be equivalent to the Stein-type strengthening  
\begin{equation}\label{stein}
\int_{\mathbb{R}^n}|\widehat{g\mathrm{d}\sigma}(x)|^2w(x)\mathrm{d}x\lesssim\int_{\mathbb{S}^{n-1}}|g(\omega)|^2\sup_{v\in\langle\omega\rangle^\perp}Xw(\omega,v)\mathrm{d}\sigma(\omega)
\end{equation}
of the global Mizohata--Takeuchi inequality \eqref{MTSglobal} in the special case that the weights are auto-correlations of nonnegative functions, that is weights of the form $w*\tilde{w}$ for $w\geq 0$. Evidently for such weights \eqref{stein} becomes
\begin{eqnarray*}
    \begin{aligned}
        \int_{\mathbb{R}^n}|\widehat{g\mathrm{d}\sigma}(x)|^2w*\tilde{w}(x)\mathrm{d}x
        &\lesssim\int_{\mathbb{S}^{n-1}}|g(\omega)|^2\sup_{v\in\langle\omega\rangle^\perp}X(w*\tilde{w})(\omega,v)\mathrm{d}\sigma(\omega).
    \end{aligned}
\end{eqnarray*}
It remains to show that the supremum in $v$ here is attained at the origin, since
$$
\int_{\mathbb{S}^{n-1}}|g(\omega)|^2X(w*\tilde{w})(\omega,0)\mathrm{d}\sigma(\omega)=\int_{\mathbb{R}^n}\left(\frac{|g(\hat{x})|^2+|g(-\hat{x})|^2}{|x|^{n-1}}\right)w*\tilde w(x)\mathrm{d}x
$$
by an application of polar coordinates. However, by Fubini's theorem and the Cauchy--Schwarz inequality,
\begin{eqnarray*}
    \begin{aligned}
        X(w*\tilde{w})(\omega,v)=\int_{\langle\omega\rangle^\perp}Xw(\omega,v')Xw(\omega,v+v')\mathrm{d}v'\leq X(w*\tilde{w})(\omega,0)
    \end{aligned}
\end{eqnarray*}
for all $v\in\langle\omega\rangle^\perp$. This special case of \eqref{stein} is
    \begin{equation}\label{2st}
    \int_{\mathbb{R}^n}|\widehat{g\mathrm{d}\sigma}(x)|^2w*\tilde{w}(x)\mathrm{d}x\lesssim\int_{\mathbb{S}^{n-1}}|g(\omega)|^2\|Xw(\omega,\cdot)\|_{L^2_v}^2\mathrm{d}\sigma(\omega)
    \end{equation}
    holding for non-negative $w$.  
    For general weights the global Stein-type inequality \eqref{stein}, being stronger than \eqref{MTSglobal}, is also false. Curiously the class of weights $\{w*\tilde{w}:w\geq 0\}$ is, formally at least, the Fourier transform of the class $\{|\widehat{h}|^2:h\geq 0\}$ for which \eqref{stein} was shown to fail in \cite{Cairo}.
\end{remark}
\begin{remark}[Bilinear formulations]\label{Rem:bil}
    One can easily ``bilinearise" Lemma \ref{tomlemma}. Suppose that $g_1,g_2$ are nonnegative functions on $\mathbb{S}^{n-1}$. If we define the  ``sup-cross-correlation" $$g_1\diamond g_2(\omega)=\sup_{\omega'}g_1(\omega')g_2(\omega''),$$
    where $\omega''$ is such that $\omega$ is a geodesic midpoint of $\omega'$ and $\omega''$, then the proof of Lemma \ref{tomlemma} reveals that
    \begin{equation}\label{bilin} g_1\mathrm{d}\sigma*\widetilde{g_2\mathrm{d}\sigma}\leq \mathcal{R}_0^*(g_1\diamond g_2).
    \end{equation}
    \end{remark}
\begin{remark}[Subspace bundle identities/bounds]
    The inequality \eqref{functionaltom}, or equivalently the positive semi-definiteness of \eqref{ftpd}, may be viewed as a spatial substitute for the phase-space (or spherical Wigner) representation \begin{equation}\label{phasespace}
    |\widehat{g\mathrm{d}\sigma}|^2=X^*W_{\mathbb{S}^{n-1}}(g,g)
    \end{equation}
    that will play a central role in our direct approach in the next subsection; see \eqref{wigna}. 
    More generally \eqref{bilin} should be contrasted with the polarised phase-space identity 
    \begin{equation}\label{phasespacebil}
    \widehat{g_1\mathrm{d}\sigma}\overline{\widehat{g_2\mathrm{d}\sigma}}=X^*W_{\mathbb{S}^{n-1}}(g_1,g_2).
    \end{equation}
    We refer to \cite[Section 4]{BGNO} for \eqref{phasespace} and its generalisations to other surfaces. The inequalities \eqref{functionaltom} and \eqref{bilin} might be referred to as ``hyperplane bundle" bounds, adapting a term from \cite{Al} in the setting of ``ray bundle" identities of the form \eqref{phasespace} and \eqref{phasespacebil}.
\end{remark}
\subsection{The direct approach via spherical Wigner distributions}\label{Subsect:wigps}
As we indicate in the introduction, our direct approach does not lead to Theorem \ref{mainsphereintro}, but rather to a variant of it. Here we prove the following generalisation of Theorem \ref{main2intro}.
\begin{theorem}\label{main2n}
Suppose that $(g_j)$ is a sequence of functions on $\mathbb{S}^{n-1}$ for which 
\begin{eqnarray}\label{kerneldef}
    \begin{aligned}
L(j,k)&:=\int_{\mathbb{S}^{n-1}}g_j(\omega')\overline{g_k(\omega')}\int_{\mathbb{S}^{n-1}}\overline{g_j(\omega)}g_k(\omega)|\omega+\omega'|^{n-3}\mathrm{d}\sigma(\omega)\mathrm{d}\sigma(\omega')\\
            &+\int_{\mathbb{S}^{n-1}}g_j(\omega')\overline{\tilde{g}_k(\omega')}\int_{\mathbb{S}^{n-1}}\overline{g_j(\omega)}\tilde{g}_k(\omega)|\omega+\omega'|^{n-3}\mathrm{d}\sigma(\omega)\mathrm{d}\sigma(\omega')
    \end{aligned}
\end{eqnarray}
is the (real-valued) kernel of an $\ell^2(\mathbb{Z})$-bounded operator, then 
\begin{equation}\label{onmtgenSn}
\sum_j\left(\int_{\mathbb{R}^n}|\widehat{g_j\mathrm{d}\sigma}|^2w\right)^2\lesssim\|Xw\|_{L^2(\{(\omega,v):\omega\in K^*, v\in \langle \omega\rangle^\perp\})}^2 
\end{equation}
for all signed weight functions $w$,
where 
\begin{equation*}
K^*=\bigcup_j\supp(g_j)^\diamond.
\end{equation*}
\end{theorem}
It will be clear from our proof that the implicit constant in \eqref{onmtgenSn} may be taken to be an explicit dimensional constant multiple of the norm of the operator with kernel $L(j,k)$.
Evidently Theorem \ref{main2n} contains Theorem \ref{main2intro} as $0\leq L(j,k)\lesssim\delta_{j,k}$ when $n=3$ under the orthonormality hypotheses of that theorem. As we clarify later, if $(g_j)$ is a sequence of spherical harmonics then the hypotheses of Theorem \ref{main2n} hold, at least in odd dimensions. We now turn to the proof of Theorem \ref{main2n}.
\begin{proof}
By duality it is equivalent to show that
\begin{equation}\label{todo}
    \int_{\mathbb{S}^{n-1}}\sum_j\lambda_j|\widehat{g_j\mathrm{d}\sigma}|^2w\lesssim\|(\lambda_j)\|_{\ell^2}\|Xw\|_{L^2(\{(\omega,v):\omega\in K^*, v\in \langle \omega\rangle^\perp\})}.
\end{equation}
We 
begin with 
the spherical Wigner transform
\begin{equation}\label{wigna}
W_{\mathbb{S}^{n-1}}(g_1,g_2)(\omega,v):=2^{n-2}\int_{\mathbb{S}^{n-1}}g_1(\omega')\overline{g_2(R_\omega\omega')}e^{-2\pi i(\omega'-R_\omega\omega')\cdot v}|\omega\cdot\omega'|^{n-2}\mathrm{d}\sigma(\omega'),
\end{equation}
where $R_\omega\omega'$ is given by \eqref{omegadef}.
 It is shown in \cite[Section 3]{BGNO} (see also \cite{Al} for the roots of this in the optics literature) that
\begin{equation}
    |\widehat{g\mathrm{d}\sigma}|^2=X^*W_{\mathbb{S}^{n-1}}(g,g),
\end{equation}
where $X^*$ is the adjoint X-ray transform. This yields
the phase-space representation
\begin{equation}
    \int_{\mathbb{S}^{n-1}}\sum_j\lambda_j|\widehat{g_j\mathrm{d}\sigma}|^2w=\int_{\mathbb{S}^{n-1}}\int_{\langle\omega\rangle^\perp}\sum_j \lambda_jW_{\mathbb{S}^{n-1}}(g_j,g_j)(\omega,v)Xw(\omega,v)\mathrm{d}v\mathrm{d}\sigma(\omega).
\end{equation}
Noting that $W_{\mathbb{S}^{n-1}}(g_j,g_j)(\omega,v)$ vanishes for $\omega\not\in K^*$, by the Cauchy--Schwarz inequality we have
$$
\int_{\mathbb{S}^{n-1}}\sum_j\lambda_j|\widehat{g_j\mathrm{d}\sigma}|^2w\leq\Bigl\|\sum_j\lambda_jW_{\mathbb{S}^{n-1}}(g_j,g_j)\Bigr\|_2\|Xw\|_{L^2(\{(\omega,v):\omega\in K^*, v\in \langle \omega\rangle^\perp\})},
$$
reducing matters to showing that
\begin{equation}\label{wigo}
\Bigl\|\sum_j\lambda_jW_{\mathbb{S}^{n-1}}(g_j,g_j)\Bigr\|_2\lesssim\|(\lambda_j)\|_{\ell^2}.
\end{equation}
Theorem~\ref{main2n} now follows from Proposition~\ref{spherical-moyal} below, which provides a spherical analogue of the classical (see for example \cite{Folland, Gosson}) Moyal identity.
\end{proof}
\begin{proposition}[Spherical Moyal Identity] 
\label{spherical-moyal}
For suitable functions $f_1, f_2$, $g_1$ and $g_2$, we have
\begin{eqnarray}\label{niceformula}
        \begin{aligned}
            \langle W_{\mathbb{S}^{n-1}}(f_1,f_2),W_{\mathbb{S}^{n-1}}(g_1,g_2)\rangle&=\frac{1}{2^{n-2}}\int_{\mathbb{S}^{n-1}}f_1(\omega')\overline{g_1(\omega')}\int_{\mathbb{S}^{n-1}}\overline{f_2(\omega)}g_2(\omega)|\omega+\omega'|^{n-3}\mathrm{d}\sigma(\omega)\mathrm{d}\sigma(\omega')\\
            &+\frac{1}{2^{n-2}}\int_{\mathbb{S}^{n-1}}f_1(\omega')\overline{\tilde{g}_2(\omega')}\int_{\mathbb{S}^{n-1}}\overline{f_2(\omega)}\tilde{g}_1(\omega)|\omega+\omega'|^{n-3}\mathrm{d}\sigma(\omega)\mathrm{d}\sigma(\omega'),
        \end{aligned}
    \end{eqnarray} 
where $\tilde{g}(\omega):=\overline{g(-\omega)}$. 
\end{proposition}
In particular, for $n=3$
\begin{equation*}
\langle W_{\mathbb{S}^{2}}(f,f),W_{\mathbb{S}^{2}}(g,g)\rangle=\frac{1}{2}\left(|\langle f,g\rangle|^2+|\langle f,\tilde{g}\rangle|^2\right).
\end{equation*}
\begin{proof}
For a fixed $\omega$ we write $\mathbb{S}^{n-1}_\pm$ for the two hemispheres constituting $\mathbb{S}^{n-1}\backslash \langle\omega\rangle^\perp$, where naturally $\omega\in\mathbb{S}^{n-1}_+$. Performing the isometric change of variables in the integral over $\mathbb{S}^{n-1}_{-}$  given by $\omega'\mapsto -R_{\omega}{\omega'}: \mathbb{S}^{n-1}_{-}\rightarrow \mathbb{S}^{n-1}_{+}$ and since $R_\omega(R_\omega\omega')=\omega'$, it is straightforward to verify that
\begin{eqnarray*}
    \begin{aligned}
        W_{\mathbb{S}^{n-1}}(g_1,g_2)(\omega,v)&=2^{n-2}\int_{\mathbb{S}^{n-1}_{+}\cup \mathbb{S}^{n-1}_{-}}g_1(\omega')\overline{g_2(R_\omega\omega')}e^{-2\pi i(\omega'-R_\omega\omega')\cdot v}|\omega\cdot\omega'|^{n-2}\mathrm{d}\sigma(\omega')\\
        &=2^{n-2}\int_{\mathbb{S}^{n-1}_+}\left(g_1(\omega')\overline{g_2(R_\omega\omega')}+\tilde{g}_2(\omega')\overline{\tilde{g}_1(R_\omega\omega')}\right)e^{-2\pi i(\omega'-R_\omega\omega')\cdot v}|\omega\cdot\omega'|^{n-2}\mathrm{d}\sigma(\omega').
    \end{aligned}
\end{eqnarray*}
As in \cite[Section 3]{BGNO} let us make the change of variables $\xi=\omega'-R_\omega\omega'$. This is a one-to-one mapping from $\mathbb{S}^{n-1}_{+}\backslash \langle\omega\rangle^\perp$ 
to the open unit ball in $\langle\omega\rangle^\perp$ (the tangent space to the sphere at $\omega$) with $\mathrm{d}\xi=2^{n-1}|\omega\cdot\omega'|\mathrm{d}\sigma(\omega')$.  Hence
\begin{eqnarray*}
    \begin{aligned}
        W_{\mathbb{S}^{n-1}}(g_1,g_2)(\omega,v)&=\frac{1}{2}\int_U \left(g_1(\omega'(\xi))\overline{g_2(R_\omega\omega'(\xi))}+\tilde{g}_2(\omega'(\xi))\overline{\tilde{g}_1(R_\omega\omega'(\xi))}\right)e^{-2\pi i\xi\cdot v}|\omega\cdot\omega'(\xi)|^{n-3}\mathrm{d}\xi,
    \end{aligned}
\end{eqnarray*}
where $\omega'(\xi)$ is the unique $\omega'$ such that $\xi=\omega'-R_\omega\omega'$. Therefore, by Parseval's identity and reversing the change of variables,
\begin{eqnarray*}
    \begin{aligned}
       \langle &W_{\mathbb{S}^{n-1}}(f_1,f_2)(\omega,\cdot),W_{\mathbb{S}^{n-1}}(g_1,g_2)(\omega,\cdot)\rangle_{L^2(\langle\omega\rangle^\perp)}\\
       &=2^{n-3}\int_{\mathbb{S}^{n-1}_+}\left(
       f_1(\omega')\overline{f_2(R_\omega\omega')}\overline{g_1(\omega')}g_2(R_\omega\omega')
+\tilde{f}_2(\omega')\overline{\tilde{f}_1(R_\omega\omega')}\overline{\tilde{g}_2(\omega')}\tilde{g}_1(R_\omega\omega')
       \right)
       |\omega\cdot\omega'|^{2n-5}\mathrm{d}\sigma(\omega')\\
       &+2^{n-3}\int_{\mathbb{S}^{n-1}_+}\left(f_1(\omega')\overline{f_2(R_\omega\omega'})\overline{\tilde{g}_2(\omega')}\tilde{g}_1(R_\omega\omega')
       +\tilde{f}_2(\omega')\overline{\tilde{f}_1(R_\omega\omega')}\overline{{g_1}(\omega')}{g_2}(R_\omega\omega')\right)
       |\omega\cdot\omega'|^{2n-5}\mathrm{d}\sigma(\omega').
    \end{aligned}
    \end{eqnarray*}
    Now,
    \begin{eqnarray*}
        \begin{aligned}
            \int_{\mathbb{S}^{n-1}_+}\tilde{f}_2(\omega')\overline{\tilde{f}_1(R_\omega\omega')}\overline{\tilde{g}_2(\omega')}\tilde{g}_1(R_\omega\omega')
        &|\omega\cdot\omega'|^{2n-5}\mathrm{d}\sigma(\omega')\\
        &=\int_{\mathbb{S}^{n-1}_+}\overline{f_2(-\omega')}f_1(-R_\omega\omega')g_2(-\omega')\overline{g_1(-R_\omega\omega')}
       |\omega\cdot\omega'|^{2n-5}\mathrm{d}\sigma(\omega')\\
       &=\int_{\mathbb{S}^{n-1}_-}\overline{f_2(R_\omega\omega')}f_1(\omega')g_2(R_\omega\omega')\overline{g_1(\omega')}
       |\omega\cdot\omega'|^{2n-5}\mathrm{d}\sigma(\omega'),
        \end{aligned}
    \end{eqnarray*}
    using the change of variables $\omega'\mapsto -R_\omega\omega'$.
    Similarly,
    \begin{eqnarray*}
        \begin{aligned}
            \int_{\mathbb{S}^{n-1}_+}\tilde{f_2}(\omega')\overline{\tilde{f_1}(R_\omega\omega')}\overline{g_1(\omega')}g_2(R_\omega\omega')
        &|\omega\cdot\omega'|^{2n-5}\mathrm{d}\sigma(\omega')\\
        &=\int_{\mathbb{S}^{n-1}_{-}}f_1(\omega')\overline{f_2(R_\omega\omega'})\overline{\tilde{g}_2(\omega')}\tilde{g}_1(R_\omega\omega')|\omega\cdot\omega'|^{2n-5}\mathrm{d}\sigma(\omega'),
        \end{aligned}
    \end{eqnarray*}
    and so,
    \begin{eqnarray*}
    \begin{aligned}
       \langle W_{\mathbb{S}^{n-1}}(f_1,f_2)(\omega,\cdot),&W_{\mathbb{S}^{n-1}}(g_1,g_2)(\omega,\cdot)\rangle_{L^2(\langle\omega\rangle^\perp)}\\&=
    2^{n-3}\int_{\mathbb{S}^{n-1}}f_1(\omega')\overline{f_2(R_\omega\omega'})\overline{g_1(\omega')}g_2(R_\omega\omega')
       |\omega\cdot\omega'|^{2n-5}\mathrm{d}\sigma(\omega')\\
       &+2^{n-3}\int_{\mathbb{S}^{n-1}}f_1(\omega')\overline{f_2(R_\omega\omega'})\overline{\tilde{g}_2(\omega')}\tilde{g}_1(R_\omega\omega')
       |\omega\cdot\omega'|^{2n-5}\mathrm{d}\sigma(\omega').
    \end{aligned}
    \end{eqnarray*}
    Hence
    \begin{eqnarray*}
        \begin{aligned}
            \langle W_{\mathbb{S}^{n-1}}&(f_1,f_2),W_{\mathbb{S}^{n-1}}(g_1,g_2)\rangle\\
            &=2^{n-3}\int_{\mathbb{S}^{n-1}}f_1(\omega')\overline{g_1(\omega')}\int_{\mathbb{S}^{n-1}}\overline{f_2(R_\omega\omega')}g_2(R_\omega\omega')|\omega\cdot\omega'|^{2n-5}\mathrm{d}\sigma(\omega)\mathrm{d}\sigma(\omega')\\
            &+2^{n-3}\int_{\mathbb{S}^{n-1}}f_1(\omega')\overline{\tilde{g}_2(\omega')}\int_{\mathbb{S}^{n-1}}\overline{f_2(R_\omega\omega')}\tilde{g}_1(R_\omega\omega')|\omega\cdot\omega'|^{2n-5}\mathrm{d}\sigma(\omega)\mathrm{d}\sigma(\omega')\\
            &=2^{n-3}\int_{\mathbb{S}^{n-1}}f_1(\omega')\overline{g_1(\omega')}\int_{\mathbb{S}^{n-1}}\overline{f_2(R_\omega\omega')}g_2(R_\omega\omega')\left|\frac{R_\omega\omega'+\omega'}{2}\right|^{n-3}|\omega\cdot\omega'|^{n-2}\mathrm{d}\sigma(\omega)\mathrm{d}\sigma(\omega')\\
            &+2^{n-3}\int_{\mathbb{S}^{n-1}}f_1(\omega')\overline{\tilde{g}_2(\omega')}\int_{\mathbb{S}^{n-1}}\overline{f_2(R_\omega\omega')}\tilde{g}_1(R_\omega\omega')\left|\frac{R_\omega\omega'+\omega'}{2}\right|^{n-3}|\omega\cdot\omega'|^{n-2}\mathrm{d}\sigma(\omega)\mathrm{d}\sigma(\omega'),
        \end{aligned}
    \end{eqnarray*}
    where we have also used \eqref{omegadef}.
The identity \eqref{niceformula} now follows on changing variables in the inner integrals from $\omega$ to $\omega''$, where $\omega''=R_\omega\omega'$, using the formula $\mathrm{d}\sigma(\omega'')=2^{n-1}|\omega\cdot\omega'|^{n-2}\mathrm{d}\sigma(\omega)$; see \cite[Section 3]{BGNO}. We clarify that the extra factor of $1/2$ needed to arrive at \eqref{niceformula} arises from the fact that the mapping $\omega\mapsto R_\omega\omega'$ is two-to-one, as clarified in the proof of Lemma \ref{tomlemma},
\end{proof}
\begin{example}[Spherical harmonics]
If the sequence $(g_j)$ is a sequence of spherical harmonics then the hypotheses of Theorem \ref{main2n} are satisfied in all \textit{odd} dimensions at least. 
We may verify this by estimating each of the two terms in \eqref{kerneldef} separately.
For the first, which we denote $\langle W(Y_j,Y_j),W(Y_k,Y_k)\rangle_+$, we develop $Y_j\overline{Y}_k$ in spherical harmonics
$$
Y_j\overline{Y_k}=\sum_{\ell}c_{jk}^{(\ell)}Y_\ell,
$$
so that
\begin{eqnarray*}
    \begin{aligned}
\langle W(Y_j,Y_j),W(Y_k,Y_k)\rangle_+&=\sum_{\ell,\ell'}c_{jk}^{(\ell)}\overline{c_{jk}^{(\ell')}}\int_{\mathbb{S}^{n-1}}Y_\ell(\omega')\int_{\mathbb{S}^{n-1}}\overline{Y_{\ell'}(\omega)}|\omega+\omega'|^{n-3}\mathrm{d}\sigma(\omega)\mathrm{d}\sigma(\omega').
\end{aligned}
\end{eqnarray*}
Then we use the Funk--Hecke theorem to observe that
$$
\int_{\mathbb{S}^{n-1}}\overline{Y_{\ell'}(\omega)}|\omega+\omega'|^{n-3}\mathrm{d}\sigma(\omega)=\alpha_n
\left(\int_{-1}^1P_{\ell'}(t)(1+t)^{\frac{n-3}{2}}(1-t^2)^{\frac{n-3}{2}}\mathrm{d}t\right)Y_\ell(\omega'),
$$
where $\alpha_n$ is some dimensional constant, and $P_\ell$ is the $\ell$th Legendre polynomial. By Rodrigues' formula this vanishes unless $|\ell'|\lesssim n$ whenever $n>1$ is odd, as the densities in the above expression are polynomial of degree $\lesssim n$. Hence by the orthonormality of the spherical harmonics,
$$
\langle W(Y_j,Y_j),W(Y_k,Y_k)\rangle_+=\sum_{|\ell|\lesssim n}|c_{kj}^{(\ell)}|^2\int_{\mathbb{S}^{n-1}}Y_\ell(\omega')\int_{\mathbb{S}^{n-1}}\overline{Y_{\ell}(\omega)}|\omega+\omega'|^{n-3}\mathrm{d}\sigma(\omega)\mathrm{d}\sigma(\omega'),
$$
and so
$$
|\langle W(Y_j,Y_j),W(Y_k,Y_k)\rangle_+|\lesssim\sum_{|\ell|\lesssim n}|c_{jk}^{(\ell)}|^2,
$$
with an implicit constant depending only on $n$. Now, by Schur's lemma it is enough to show that
$$
\sum_k\sum_{|\ell|\lesssim n}|c_{jk}^{(\ell)}|^2
$$
is bounded uniformly in $j$. To see this we observe that 
$$
c_{jk}^{(\ell)}=\int_{\mathbb{S}^{n-1}}Y_j\overline{Y_k}\overline{Y_\ell}=\langle Y_j\overline{Y}_\ell,Y_k\rangle,
$$
and so by Bessel's inequality,
$$
\sum_k\sum_{|\ell|\lesssim n}|c_{jk}^{(\ell)}|^2\leq\sum_{|\ell|\lesssim n}\|Y_j\overline{Y_\ell}\|_2^2\leq\sum_{|\ell|\lesssim n}\|Y_\ell\|_\infty^2,
$$
which is a finite constant independent of $j$. A similar argument applies to the second term in \eqref{kerneldef}. Evidently one could instead appeal to Theorem \ref{mainsphereintro} here, which is applicable in both odd and even dimensions.
\end{example}
\section{More general submanifolds}\label{Sect:gen}
For a given $1\leq p<\infty$ one might reasonably ask: under what conditions on a sequence of functions $(g_j)$ on a submanifold $S$ is it true that
\begin{equation}\label{onmtgen}
\sum_j\left(\int_{\mathbb{R}^n}|\widehat{g_j\mathrm{d}\sigma}|^2w\right)^p\lesssim\|X_Sw\|_{L^p(\{(u,v)\in TS:u\in E\})}^p
\end{equation}
for some geometrically-meaningful set $E\subseteq S$ containing $$K=\bigcup_j\supp(g_j)?$$
Here $\mathrm{d}\sigma$ denotes surface measure on $S$, and $X_S$ is the X-ray transform pulled back by the gauss map $N:S\rightarrow\mathbb{S}^{n-1}$, that is, $X_Sw(u,v)=Xw(N(u),v)$ where $(u,v)\in TS$.
As in the case of the sphere, \eqref{onmtgen} is easily seen to hold when $(g_j)$ is an orthonormal sequence in $L^2(S)$ and $E=K$ when $p=1$, and this follows by a straightforward application of Bessel's inequality.
In this section we provide some answers to this question when $p=2$, establishing results that may be interpreted as generalisations of Theorems \ref{mainsphereintro} and \ref{main2n}. While both approaches are effective, we shall see that the Schatten space perspective produces statements that are simpler to interpret.

The first thing to observe is that \eqref{onmtgen} is quickly established for all $p$ when $S$ is a hyperplane and $E=K$, clarifying that a nonvanishing curvature hypothesis on $S$ should not be decisive. If $S=\mathbb{R}^{n-1}\times\{0\}$ then \eqref{onmtgen} 
becomes a statement about the Fourier transform (a certain weighted inequality), specifically
\begin{equation}\label{hyp}
\sum_j\left(\int_{\mathbb{R}^{n-1}}|\widehat{g}_j(v)|^2Xw(e_n,v)\mathrm{d}v\right)^p\lesssim |K|\|Xw(e_n,\cdot)\|_{L^p(\mathbb{R}^{n-1})}^p,
\end{equation}
or equivalently
\begin{equation}\label{hyp'}
\int_{\mathbb{R}^{n-1}}\sum_j\lambda_j|\widehat{g}_j(v)|^2Xw(e_n,v)\mathrm{d}v\lesssim \|(\lambda_j)\|_{\ell^{p'}}|K|^{1/p}\|Xw(e_n,\cdot)\|_{L^p(\mathbb{R}^{n-1})}.
\end{equation}
Since $Xw(e_n,\cdot)$ is an arbitrary weight on $\mathbb{R}^{n-1}$, by duality \eqref{hyp'} is equivalent to
\begin{equation}\label{hyp''}
   \Bigl\|\sum_j\lambda_j |\widehat{g}_j|^2\Bigr\|_{p'}\lesssim \|(\lambda_j)\|_{\ell^{p'}}|K|^{1/p}.
   \end{equation}
To see this (and with constant $1$) 
we fix an orthonormal sequence $(g_{j})$ and define the operator $T$ acting on complex sequences $\lambda=(\lambda_{j})$ by $$T(\lambda)(\xi)=\sum_{j}\lambda_{j}|\widehat{g}_{j}(\xi)|^{2}.$$ For this we have the estimates
$$\|T(\lambda)\|_{L^1(\mathbb{R}^{n-1})}\leq \|\lambda\|_{\ell^1},\quad \|T(\lambda)\|_{L^\infty(\mathbb{R}^{n-1})}\leq |K| \|\lambda\|_{\ell^\infty},$$
the former being trivial and the latter being a simple consequence of Bessel's inequality. The inequality \eqref{hyp''} now follows by interpolation.
It is appropriate to caution that the case of a hyperplane does not necessarily constitute strong evidence in support of the truth of \eqref{onmtgen} with $E=K$, especially in light of the recent counterexample to \eqref{MTSglobal} in \cite{Cairo}. In particular, the counterexamples of Cairo and Zhang \cite{CZ} provide compact convex $C^2$ hypersurfaces for which \eqref{onmtgen} fails for $p>n+1$, even when $E=S$! See Remark \ref{Rem:CZ}.
\begin{remark}\label{remscal}
A further pertinent observation, not unrelated to the previous one, is that \eqref{onmtgen} is invariant under isotropic scalings. Concretely, if \eqref{onmtgen} holds for $S$ then it holds for any isotropic dilate $RS$ of $S$, and with the same implicit constant. This is a simple exercise using the definition of the measure
$$
\int_{TS}\varphi:=\int_S\int_{T_uS}\varphi(u,v)\mathrm{d}v\mathrm{d}\sigma(u)
$$
on the tangent bundle $TS$; see \cite{BBC3} for a similar remark in the context of the original Mizohata--Takeuchi conjecture \eqref{MTSglobal}. Consequently progress on \eqref{onmtgen} will not naturally refer to any lower bound on the curvature of $S$, in contrast with the more conventional linear extension problems formulated with Lebesgue measures.
\end{remark}
\subsection{A class of hypersurfaces}
For reasons that will quickly become apparent, in this section we restrict our attention to strictly convex hypersurfaces $S$ in $\mathbb{R}^n$. 
In order to extend Theorem \ref{mainsphereintro} to such $S$ we shall need to define an analogue of the ``spherical midpoint set" for a general subset of such a hypersurface $S$.  
To this end we use the convexity of $S$ to define, for $u'\neq u$, the point $R_u u'$ to be the unique $u''\in S$ such that
\begin{equation}\label{collision condition 1}
(u'-u'')\cdot N(u)=0
\end{equation}
and 
\begin{equation}\label{collision condition 2}
N(u)\wedge N(u')\wedge N(u'')=0.
\end{equation}
For completeness we define $R_{u}u$ to be $u$ for all points $u$. So that the map $u\mapsto R_uu'$ is surjective we assume in addition that the normal set $N(S)\subseteq\mathbb{S}^{n-1}$ is geodesically convex. Now, for an arbitrary $E\subseteq S$ we may then define
\begin{eqnarray}\label{genmid}
\begin{aligned}
E^\diamond&=\{u\in S:  (u'-u'')\cdot N(u)=0, N(u)\wedge N(u')\wedge N(u'')=0\mbox{ for some } u',u''\in E\}\\&=\{u\in S: u''=R_u u'\mbox{ for some } u',u''\in E\}.
\end{aligned}
\end{eqnarray}
If $S=\mathbb{S}^{n-1}$ then this is precisely the geodesic midpoint set of $E$ as defined in the introduction. In general the normal set $N(E^\diamond)$ of $E^\diamond$ is a certain set of ``in-between points" of pairs of points from the normal set $N(E)$ of $E$.

Our arguments will require that we impose some further structural conditions on the hypersurfaces $S$, starting with the assumption that it is \textit{a smooth graph} (meaning that our results will not logically imply those for the sphere in the previous section). With Remark \ref{remscal} in mind it is natural to seek sufficient conditions on $S$ that are invariant under isotropic scalings. A convenient condition, that may be thought of as a certain strong form of quasi-conformality of the shape operator of $S$, was considered recently in \cite{BGNO}, and will also be suitable here. We define the \textit{curvature quotient} $Q(S)$ of $S$ to be the maximum ratio of the principal curvatures of $S$, namely the smallest constant $c$ such that
$
\lambda_j(u)\leq c\lambda_{k}(u')
$
for all $u, u'\in S$ and $1\leq j,k\leq n-1$, where $\lambda_j(u)$ denotes the $j$th principal curvature of $S$ at the point $u$. For the rest of this section we suppose that $Q(S)$ is finite.
\subsection{The Schatten space approach}  The main purpose of this section is to establish the following:
 \begin{theorem}\label{mainn-genS}
Suppose $n\geq 2$ and that $S$ has finite curvature quotient $Q(S)$. Then
\begin{equation}\label{onmtgenSn-1-genS}
\sum_j\left(\int_{\mathbb{R}^n}|\widehat{g_j\mathrm{d}\sigma}|^2w\right)^2\lesssim\|X_Sw\|_{L^2(\{(u,v)\in TS:u\in K^\diamond\})}^2
\end{equation}
for all orthonormal sequences $(g_j)$ in $L^2(S)$ and signed weight functions $w$,
where 
\begin{equation}\label{defK-S}
K=\bigcup_j\supp(g_j).
\end{equation}
    \end{theorem}
Our proof of Theorem \ref{mainn-genS} proceeds very much as in the spherical case of Section \ref{Sect:sp}, and so we only elaborate on the elements that are sensitive to the geometry of $S$. As we shall see, the implicit constant in \eqref{onmtgenSn-1-genS} will be seen to be at most a certain power of $Q(S)$, up to a dimensional constant factor. It is enough to prove 
\begin{equation}\label{e:Thm2-Support-S}
\| \mathcal{E}_K^* w \mathcal{E}_K \|_{\mathcal{C}^2(L^2(S))}
\lesssim 
\|X_Sw\|_{L^2(\{(u,v)\in TS:u\in K^\diamond\})},
\end{equation}
where $\mathcal{E}_Kg:=\mathcal{E}(\mathbf{1}_Kg)$ and $\mathcal{E}g=\widehat{g\mathrm{d}\sigma}$.
Arguing as in the spherical case, the inequality \eqref{e:Thm2-Support-S} is seen to be equivalent to the pointwise inequality
$$\mathrm{d}\sigma_K*\widetilde{\mathrm{d}\sigma}_K(\xi)\lesssim\int_{K^{
\diamond}}\delta(N(u)\cdot\xi)\mathrm{d}\sigma(u),$$
or equivalently,
$$\mathrm{d}\sigma_K*\widetilde{\mathrm{d}\sigma}_K\lesssim\mathcal{R}_{S,0}^*(\mathbf{1}_{K^\diamond})$$
where $\mathcal{R}_{S,0}f(u):=\mathcal{R}_0f(N(u))$ and the restricted Radon transform $\mathcal{R}_0$ is given by \eqref{resrad}.
This is a direct consequence of the following lemma applied to $g=\mathbf{1}_K$.
\begin{lemma}[Tomographic estimate: general hypersurfaces]\label{tomlemma-generalS} Suppose $n\geq 2$.
    For a suitable $g:S\rightarrow\mathbb{R}_+$ define $g^\diamond:S\rightarrow\mathbb{R}_+$ by the ``sup-autocorrelation" formula
    $$
    g^\diamond(u)=\sup_{u'}g(u')g(R_{u}u').
    $$
    Then there exists a constant $c$ depending only on $n$ such that
    \begin{equation}\label{functionaltom-genS}
    g\mathrm{d}\sigma*\widetilde{g\mathrm{d}\sigma}\leq cQ(S)^{\frac{5n-8}{2}}
    \mathcal{R}_{S,0}^*g^\diamond.
    \end{equation}    
    \end{lemma}
    \begin{proof} It suffices to prove that
      \begin{equation}\label{weak2}
      \int_{S}\int_{S}\varphi(u'-u'')g(u')g(u'')\mathrm{d}\sigma(u')\mathrm{d}\sigma(u'')\leq cQ(S)^{\frac{5n-8}{2}}\int_{S}g^\diamond(u)\mathcal{R}_0\varphi(N(u))\mathrm{d}\sigma(u)
      \end{equation}
      for all nonegative test functions $\varphi$ on $\mathbb{R}^n$. Let $J(u,u')$ be the reciprocal of the Jacobian of the mapping $u\mapsto R_uu'$, so that 
\begin{equation}\label{defJ}
\int_S\Phi(R_uu')J(u,u')\mathrm{d}\sigma(u)=\int_{S}\Phi \mathrm{d}\sigma
\end{equation}
for each $u'\in S$. For each $u'$ we make the change of variables $u''=R_u u'$, which gives
      \begin{eqnarray*}
          \begin{aligned}
      \int_{S}\int_{S}\varphi(u'-u'')g(u')g(u'')\mathrm{d}\sigma(u')\mathrm{d}\sigma(u'')=\int_{S}\int_{S}\varphi(u'-R_u u')g(u')g(R_u u')J(u,u')\mathrm{d}\sigma(u')\mathrm{d}\sigma(u).
      \end{aligned}
      \end{eqnarray*}
      Next, for fixed $u$, let $\xi=u'-R_u u'$ and let $\widetilde{J}(u,u')$ be the Jacobian of the map $u'\mapsto\xi$. In \cite[Section 4]{BGNO} it is proved that
      $$\frac{J(u,u')}{\widetilde{J}(u,u')}\lesssim  Q(S)^{\frac{5n-8}{2}}$$ with a suitable implicit constant, and so
      \begin{eqnarray*}
          \begin{aligned}
      \int_{S}\varphi(u'-R_u u')g(u')g(R_u u')J(u,u')\mathrm{d}\sigma(u')&=\int_{T_uS}\varphi(\xi)g(u'(\xi))g(R_u u'(\xi))\frac{J(u,u')}{\widetilde{J}(u,u')}\mathrm{d}\xi\\
      &\leq cQ(S)^{\frac{5n-8}{2}}g^\diamond(u)\mathcal{R}_0\varphi(N(u)),
      \end{aligned}
      \end{eqnarray*}
      as required.
    \end{proof}
    \begin{remark}
        As will be expected one may extend Remark \ref{Rem:bil} to this setting as follows: if for nonnegative functions $g_1,g_2$ on $S$ we define $g_1\diamond g_2$ on $S$ by the ``sup-cross-correlation" formula
        $$
        g_1\diamond g_2(u)=\sup_{u'\in S}g_1(u')g_2(R_uu'),
        $$
        then $$g_1\mathrm{d}\sigma*\widetilde{g_2\mathrm{d}\sigma}\leq cQ(S)^{\frac{5n-8}{2}}\mathcal{R}_{S,0}^*(g_1\diamond g_2)$$
        where $\mathcal{R}_{S,0}$ is the Radon transform restricted to hyperplanes through the origin, pulled back by the gauss map of $S$. It is interesting to contrast this with the phase-space representation \eqref{genphasesp} that is central to our direct approach in the next subsection.
    \end{remark}
    \begin{remark}\label{muchbetter}
        There is no suggestion with Lemma \ref{tomlemma-generalS} or Theorem \ref{mainn-genS} that the finiteness of $Q(S)$ is necessary for the claimed inequalities to hold. For example, when $n=2$ better estimates are available thanks to the explicit formula \eqref{Lambda}. In that case we have
        $$
        g_1\mathrm{d}\sigma*\widetilde{g_2\mathrm{d}\sigma}\leq\Lambda(S)^2\mathcal{R}_{S,0}^*(g_1\diamond g_2),
        $$
        where 
        $$
        \Lambda(S):=\sup_{u, u'\in S}\left(\frac{|u''-u'|K(u)}{|N(u')\wedge N(u'')|}\right)^{1/2}.
        $$
        For example, $\Lambda(S)<\infty$ when $S=\{(t,t^4): t\in [-1,1]\}$; see \cite{BGNO}, where the quantity $\Lambda(S)$ also arises naturally.
    \end{remark}
    \subsection{The direct approach via \texorpdfstring{$S-$c} carried Wigner distributions}
    The direct approach used in Section \ref{Subsect:wigps} applies at this level of generality thanks to the phase-space representation
    \begin{equation}\label{genphasesp}
    |\widehat{g\mathrm{d}\sigma}|^2=X_S^*W_S(g,g)
    \end{equation}
    established in \cite{BGNO}. Here
    $$
    W_S(g_1,g_2)(u,v):=\int_S g_1(u')\overline{g_2(R_uu')}e^{-2\pi iv\cdot(u'-R_uu')}J(u,u')\mathrm{d}\sigma(u'),
    $$
    with $J$ defined by \eqref{defJ}, is what we refer to as the $S$-carried Wigner transform; see \cite{Al} for the origins of this in optics.
As we have seen in previous sections, we will have the inequality
$$
\sum_j\left(\int_{\mathbb{R}^n}|\widehat{g_j\mathrm{d}\sigma}|^2w\right)^2\lesssim\|X_Sw\|_{L^2(\{(u,v)\in TS:u\in K^*\})}^2
$$
for all (signed) weight functions $w$
provided $(g_j)$ is a sequence of functions on $S$ for which
\begin{equation}\label{need}
L(j,k):=\langle W_S(g_j,g_j),W_S(g_k,g_k)\rangle
\end{equation}
is the kernel of an $\ell^2$-bounded operator. Here $$K^*=\bigcup_j\supp(g_j)^\diamond\subseteq K^\diamond.$$
In order to be able to clarify this almost orthonormality hypothesis we will need a suitable Moyal-type identity for the $S$-carried Wigner distribution. Arguing as in the spherical case one quickly finds that
$$
\langle W_S(f,f), W_S(g,g)\rangle=\int_Sf(u')\overline{g(u')}\left(\int_S \overline{f(R_uu')}g(R_uu')\frac{J(u,u')}{\widetilde{J}(u,u')}J(u,u')\mathrm{d}\sigma(u)\right)\mathrm{d}\sigma(u'),
$$
so that if we let $u''=R_uu'$ and $$M(u',u'')=\frac{J(u,u')}{\widetilde{J}(u,u')},$$
then
\begin{equation}\label{Moyal-gen}
\langle W_S(f,f), W_S(g,g)\rangle=\int_Sf(u')\overline{g(u')}\left(\int_S \overline{f(u'')}g(u'')M(u',u'')\mathrm{d}\sigma(u'')\right)\mathrm{d}\sigma(u');
\end{equation}
here $\widetilde{J}$ is defined as in the proof of Lemma \ref{tomlemma-generalS}.
In defining $M$ in this way we are implicitly writing $u$ in terms of $u'$ and $u''$, which we may do by the assumed geodesic convexity of $N(S)$.
There is a simple formula for $M$ when $n=2$ at least. In this case elementary trigonometry reveals that
\begin{eqnarray}\label{jtilde}
\begin{aligned}
\widetilde{J}(u,u')&=N(u)\cdot N(u')+N(u)\cdot N(u'')\frac{|N(u)\wedge N(u')|}{|N(u)\wedge N(u'')|}\\
&=\frac{N(u)\cdot N(u')|N(u)\wedge N(u'')|+N(u)\cdot N(u'')|N(u)\wedge N(u')|}{|N(u)\wedge N(u'')|}\\
&=\frac{|N(u')\wedge N(u'')|}{|N(u)\wedge N(u'')|}.
\end{aligned}
\end{eqnarray}
From \cite{BGNO} we know that
$$
J(u,u')=\frac{|u''-u'|}{|N(u)\wedge N(u'')|}K(u),
$$
where $K(u)$ denotes the gaussian curvature of $S$ at the point $u$, and so
\begin{equation}\label{Lambda}
M(u',u'')=\frac{|u''-u'|K(u)}{|N(u')\wedge N(u'')|}.
\end{equation}
A review of our argument reveals the more general Moyal identity
\begin{equation*}\label{Moyal-genn=2'}
\langle W_S(f_1,f_2), W_S(g_1,g_2)\rangle=\int_Sf_1(u')\overline{g_1(u')}\left(\int_S \overline{f_2(u'')}g_2(u'')\frac{|u''-u'|K(u)}{|N(u')\wedge N(u'')|}\mathrm{d}\sigma(u'')\right)\mathrm{d}\sigma(u').
\end{equation*}
It seems likely that an explicit formula for $M$ may be derived in terms of $N$ and its derivatives in all dimensions using ideas from \cite{BGNO}, where such a formula is derived for $J$.
\section{The paraboloid: orthonormal weighted Strichartz estimates}\label{Sect:par}
The main purpose of this section is to prove Theorem \ref{main}. 
As we mention in the introduction, the Wigner distribution approach appears to be more effective than the Schatten space approach in this setting. We begin with the latter.
\subsection{The Schatten approach}
Here we establish a weakened form of Theorem \ref{main} where $M$ is replaced with the larger $\mathbb{R}^d\times K^\diamond$, where $K^\diamond=\frac{1}{2}(K+K)$ and $$K=\bigcup_j\supp(\widehat{u}_{j,0}).$$ 
This argument is simpler than its spherical counterpart in Section \ref{Sect:sp} as the analogue of Lemma \ref{tomlemma} turns out to be essentially tautological.

For $f\in L^2(\mathbb{R}^d)$ let
$$S_Kf(x,t)=\int_{\mathbb{R}^d}e^{-2\pi i(x\cdot\xi+t|\xi|^2)}\mathbf{1}_K(\xi)\widehat{f}(\xi)\mathrm{d}\xi$$ be the solution to the free Schr\"odinger equation \eqref{schrod} with initial data $P_Kf$, the frequency projection of $f$ onto $K$. Evidently $u_{j}=S_Ku_{j,0}$ for each $j$ by hypothesis.
It suffices to show that
\begin{equation}\label{shpar}
\|S_K^*wS_K\|_{\mathcal{C}^2}\leq \|\rho^*w\|_{L^2(\mathbb{R}^d\times K^\diamond)}.
\end{equation}
By Plancherel's theorem in the first variable it is quickly verified that
$$
\|\rho^*w\|_{L^2(\mathbb{R}^d\times K^\diamond)}^2=\int_{\mathbb{R}}\int_{\mathbb{R}^d}|\widehat{w}(\xi, \tau)|^2\mathfrak{R}_0^*\mathbf{1}_{K^\diamond}(\xi,\tau)\mathrm{d}\xi\mathrm{d}\tau
$$
where $$\mathfrak{R}_0h(v):=\int_{\mathbb{R}^{d+1}}h(\xi,\tau)\delta(\tau-2v\cdot\xi)\mathrm{d}\xi\mathrm{d}\tau$$
is a certain (parametrised) space-time Radon transform restricted to space-time hyperplanes passing through the origin. We use this notation to emphasise the similarity with the spherical case in Section \ref{Sect:sp}. It remains to observe that
$$
\|S_K^*wS_K\|_{\mathcal{C}^2}^2=\int_{\mathbb{R}}\int_{\frac{1}{2}(K-K)}|\widehat{w}(\xi, \tau)|^2\mathfrak{R}_0^*\mathbf{1}_{K^\diamond}(\xi,\tau)\mathrm{d}\xi\mathrm{d}\tau,
$$
which follows from the identity
$$
\|S_K^*wS_K\|_{\mathcal{C}^2}^2=\int_K\int_K|\widehat{w}(\xi-\eta,|\xi|^2-|\eta|^2)|^2\mathrm{d}\xi\mathrm{d}\eta
$$
by little more than a change of variables. It might be interesting to try to adapt this approach to yield a Fourier-invariant statement similar to that of Theorem \ref{main}.
\begin{remark}\label{Remcopospara}
As the above argument reveals, in order to establish \eqref{onmtconj} as stated for $p=2$, it would suffice to show that
$$
\int_{\mathbb{R}}\int_{\frac{1}{2}(K-K)}|\widehat{w}(\xi, \tau)|^2\mathfrak{R}_0^*\mathbf{1}_{K^\diamond}(\xi,\tau)\mathrm{d}\xi\mathrm{d}\tau\lesssim \int_{\mathbb{R}}\int_{\mathbb{R}^d}|\widehat{w}(\xi, \tau)|^2\mathfrak{R}_0^*\mathbf{1}_{K}(\xi,\tau)\mathrm{d}\xi\mathrm{d}\tau
$$
for all $K$ and all \textit{nonnegative} $w$. This is a co-positivity statement in the sense of Remark \ref{Remcopos}.
\end{remark}
\subsection{The Wigner approach} As we shall see, the appropriate Wigner distribution in this setting is the classical one, whose behaviour on $L^2$ is particularly straightforward. This simplifies our analysis considerably. 
In order to prove Theorem \ref{main}, by duality it suffices to show that
\begin{equation}\label{dualonmt}
\int_{\mathbb{R}^{d+1}}\sum_j\lambda_j|u_j|^2w\mathrm{d}x\mathrm{d}t\leq\|(\lambda_j)\|_{\ell^2}\|\rho^*w\|_{L^2(M)}
\end{equation}
for all sequences $(\lambda_j)\in\ell^2$. We begin by applying the classical observation (originating in \cite{Wigner}; see \cite{BGNO} for clarification) that for any solution $u$ of the Schr\"odinger equation \eqref{schrod}, 
$$
|u|^2=\rho(W(u_0,u_0)),
$$
where $W(u_0,u_0)$ is the Wigner distribution \eqref{eucwig} of $u_0$, and $\rho$ is the ``velocity averaging'' operator
$$
\rho(f)(x,t)=\int_{\mathbb{R}^d}f(x+tv,v)\mathrm{d}v.
$$
This is a simple consequence of the basic fact that the Wigner distribution of a solution to a free Schr\"odinger equation satisfies a free kinetic transport equation.
By linearity this leads to the phase-space representation
\begin{equation}\label{thanks to Eugene P Wigner}
\sum_j\lambda_j|u_j|^2=\rho\Bigl(\sum_j\lambda_jW(u_{j,0},u_{j,0})\Bigr),
\end{equation} which we use to write
\begin{equation}\label{ps}
\int_{\mathbb{R}^{d+1}}\sum_j\lambda_j|u_j|^2w\mathrm{d}x\mathrm{d}t=\int_{\mathbb{R}^{2d}}\sum_j \lambda_jW(u_{j,0},u_{j,0})(x,v)\rho^*w(x,v)\mathrm{d}x\mathrm{d}v.
\end{equation}
Using the Fourier-invariance property
$$
W(u_0,u_0)(x,v)=W(\widehat{u}_0,\widehat{u}_0)(v,x)=\int_{\mathbb{R}^d}\widehat{u}_0\left(v+\frac{y}{2}\right)\overline{\widehat{u}_0\left(v-\frac{y}{2}\right)}e^{-2\pi ix\cdot y}\mathrm{d}y,
$$
it follows that the support of $W(u_{j,0},u_{j,0})$ is contained in $\supp(u_{j,0})^\diamond\times\supp(\widehat{u}_{j,0})^\diamond$ for each $j$, and so
it remains to apply the Cauchy--Schwarz inequality in \eqref{ps} and use the orthonormality of $(u_{j,0})$ via the  classical Moyal identity
\begin{equation}\label{classicalmoyal}
\langle W(f_1,f_2),W(g_1,g_2)\rangle=\langle f_1,g_1\rangle\overline{\langle f_2,g_2\rangle};
\end{equation}
the point being that $(W(u_{j,0},u_{j,0}))$ is orthonormal whenever $(u_{j,0})$ is; see for example \cite{Folland, Gosson}. This completes the proof of Theorem \ref{main}.
\begin{remark}[Almost orthonormal Strichartz estimates]\label{Rem:ao}
It is evident that the orthonormality of the sequence $(u_{j,0})$ may be relaxed considerably in Theorem \ref{main}. All that our proof requires is that 
$$
K(j,k):=|\langle u_{j,0},u_{k,0}\rangle|^2
$$
is the kernel of an $\ell^2$ bounded operator. For example, by Schur's test it suffices that
$$
\sup_k\sum_j|\langle u_{j,0},u_{k,0}\rangle|^2<\infty.
$$
\end{remark}


\section{Observations for \texorpdfstring{$p\not=2$}{p!=2}}\label{Sect:remarks}
So far essentially all of our results have concerned the case $p=2$ of \eqref{p} and its variants for submanifolds other that the sphere. In this section we collect together a number of observations pertaining to $p\not=2$. These naturally fall into three categories: $1\leq p<2$, $p>2$ and $p<1$.
\subsection{Interpolation for \texorpdfstring{$1\leq p\leq 2$}{1<= p<= 2}}\label{gapfill} 
Given that we have been able to establish \eqref{p} and its variants in some form for both $p=1$ and $p=2$, it is natural to expect to be able to effectively interpolate in order to reach $1<p<2$. While we have been unable to do this, we are able to establish some natural alternative (Sobolev) interpolants of the $p=1$ and $p=2$ cases of \eqref{p}. The purpose of this brief section is to clarify this in the particular setting of the sphere, although we expect similar arguments to apply more generally.
Here we consider the weaker undirected form \eqref{undirp} of \eqref{p}, where the restriction of the X-ray transform to $K$ is dropped from the right-hand side. 
\begin{proposition}\label{interpolants}
    If $1\leq p\leq 2$ then
    \begin{equation}\label{pp}
\sum_j\left|\int_{\mathbb{R}^n}|\widehat{g_j\mathrm{d}\sigma}|^2w\right|^p\lesssim\|(-\Delta)^{-\frac{1}{2p'}}w\|_p^p.
\end{equation}
\end{proposition} 
We clarify that signed (and thus complex-valued) weights $w$ are permitted in the statement of Proposition \ref{interpolants}.
Since the inequality
\begin{equation}\label{HLSXaa} 
    \|Xf\|_p\lesssim\|(-\Delta)^{-\frac{1}{2p'}}f\|_p
\end{equation}
holds whenever $1\leq p\leq 2$ (see \cite{Strich}), the inequality \eqref{pp} is weaker than the intended \eqref{undirp}, at least for non-negative $w$. As may be expected from wave-packet heuristics, which identify $\|Xw\|_p^p$ as a certain proxy for the left-hand side of \eqref{pp}, Proposition \ref{interpolants} may be established by an interpolation argument in much the same way as \eqref{HLSXaa}.
\begin{proof}
Proposition \ref{interpolants} will follow by a straightforward application of Theorem \ref{mainsphereintro} and the Fefferman--Stein interpolation theorem for analytic families of operators. To this end let us fix an orthonormal sequence $(g_j)$, and for $z\not=-1$ let 
$$
T_zw(j)=(z+1)^{-n-1}\int_{\mathbb{R}^n}|\widehat{g_j\mathrm{d}\sigma}(x)|^2(-\Delta)^zw(x)\mathrm{d}x.
$$
Recalling the pointwise inequality \eqref{preFS}, it follows from well-known Hardy space bounds (see for example \cite{FSt}) on the Marcinkiewicz operator $(-\Delta)^{it}$ that
\begin{equation}\label{p1}
\|T_{it}w\|_{\ell^1}\lesssim (1+|t|)^{-n-1}\|(-\Delta)^{it}w\|_1\lesssim\|w\|_{H^1}
\end{equation}
uniformly in $t\in\mathbb{R}$.
Further, by Theorem \ref{mainsphereintro} applied to the (complex-valued) weight $(-\Delta)^{it}w$,
\begin{equation}\label{p2}
\|T_{\frac{1}{4}+it}w\|_{\ell^2}\lesssim\|(-\Delta)^{it}w\|_2=\|w\|_2
\end{equation}
for all $t$, and so an application of the Fefferman--Stein analytic interpolation theorem \cite{FSt} yields
\begin{equation}\label{canget}\left\|\left(\int_{\mathbb{R}^n}|\widehat{g_j\mathrm{d}\sigma}|^2(-\Delta)^{\frac{1}{2p'}} w\right)\right\|_{\ell^p}=\|T_{\frac{1}{2p'}}w\|_{\ell^p}\lesssim\|w\|_p,
\end{equation}
or equivalently \eqref{pp}
for all $1< p\leq 2$. The case $p=1$ is an immediate consequence of \eqref{preFS}. 
\end{proof}
\subsection{Higher \texorpdfstring{$p$}{p} and the co-positivity of tensor forms}\label{pten}
While we do not provide any results towards \eqref{p} and its variants for $p>2$, the Schatten space approach to \eqref{p} from Section \ref{Sect:sp} may be readily developed in the case that $p$ is an even integer. This allows \eqref{p} and its variants to be reduced to the control of one $p$-tensor form by another, or more accurately, the co-positivity (or positivity if we hope to allow signed $w$) of a $p$-tensor form; this is already evident in Sections \ref{Sect:sp}--\ref{Sect:par} in the relatively simple case $p=2$ -- see Remarks \ref{boch}--\ref{impst}. In this section we expose some of the higher tensor forms that emerge in the context of the sphere, although similar reasoning may be applied in other contexts.

As in Section \ref{Sect:sp}, we observe that the estimate \eqref{p}, or rather, for the sake of simplicity, its undirected form \eqref{undirp},
may be equivalently expressed as
   \begin{equation} \label{punshat}
\| \mathcal{E}^* w \mathcal{E} \|_{\mathcal{C}^p}^p \lesssim \|Xw\|_{L^p}^p,
\end{equation} 
where $\mathcal{E}g:=\widehat{g\mathrm{d}\sigma}$. Here $w$ is a real-valued weight so that $\overline{\widehat{w}(\xi)}=\widehat{w}(-\xi)$.
By way of an example, when $p=4$ routine calculations reveal that
\begin{eqnarray*}\label{shat4}
\begin{aligned}
\| \mathcal{E}^* &w \mathcal{E} \|_{\mathcal{C}^4}^4
\\&=\int_{(\mathbb{S}^{n-1})^2}|\ker((\mathcal{E}^*w\mathcal{E})^2)(\omega,\omega')|^2\mathrm{d}\sigma(\omega)\mathrm{d}\sigma(\omega')\\
&=\int_{(\mathbb{S}^{n-1})^2}\Bigl|\int_{\mathbb{S}^{n-1}}\widehat{w}(\omega-\omega'')\overline{\widehat{w}(\omega''-\omega')}\mathrm{d}\sigma(\omega'')\Bigr|^2\mathrm{d}\sigma(\omega)\mathrm{d}\sigma(\omega')\\
&=\int_{(\mathbb{S}^{n-1})^4}\widehat{w}(\omega_1-\omega_2)\widehat{w}(\omega_2-\omega_3)\widehat{w}(\omega_3-\omega_4)\widehat{w}(\omega_4-\omega_1)\mathrm{d}\sigma(\omega_1)\mathrm{d}\sigma(\omega_2)\mathrm{d}\sigma(\omega_3)\mathrm{d}\sigma(\omega_4)\\& = \int_{(\mathbb{R}^n)^4} w(x_1)w(x_2)w(x_3) w(x_4)  
   \widehat{\mathrm{d}\sigma}(x_1-x_2)\widehat{\mathrm{d}\sigma}(x_2-x_3)\widehat{\mathrm{d}\sigma}(x_3-x_4)\widehat{\mathrm{d}\sigma}(x_4-x_1) \mathrm{d}x_1 \mathrm{d}x_2 \mathrm{d}x_3 \mathrm{d}x_4,
\end{aligned}
\end{eqnarray*}
and 
\begin{eqnarray}\label{X4}
\begin{aligned}
\|Xw\|_4^4
&=\int_{(\mathbb{R}^n)^2}\frac{w(x_1)w(x_2)(Xw(\ell(x_1,x_2)))^2}{|x_1-x_2|^{n-1}}\mathrm{d}x_1\mathrm{d}x_2\\
&=\int_{(\mathbb{R}^n)^4} w(x_1)w(x_2)w(x_3)w(x_4)  \frac{\delta_{\ell(x_1,x_2)}(x_3)\delta_{\ell(x_1,x_2)}(x_4)}{|x_1-x_2|^{n-1}}\mathrm{d}x_1 \mathrm{d}x_2 \mathrm{d}x_3 \mathrm{d}x_4,
\end{aligned}
\end{eqnarray}
where $\ell(x_1,x_2)$ is the line in $\mathbb{R}^n$ passing through both $x_1$ and $x_2$, and for a line $\ell$,
$$
\langle \delta_\ell,\varphi\rangle:=X\varphi(\ell).
$$
This expression for $\|Xw\|_4^4$ is a case of Drury's identity \cite{Drury}; see also \cite{Baern}.
Thus the inequality \eqref{punshat}, if true for nonnegative weights $w$, is the statement that 
\begin{equation}\label{nontrivi}
k_4(x):=\frac{\delta_{\ell(x_1,x_2)}(x_3)\delta_{\ell(x_1,x_2)}(x_4)}{|x_1-x_2|^{n-1}}-c\widehat{\mathrm{d}\sigma}(x_1-x_2)\widehat{\mathrm{d}\sigma}(x_2-x_3)\widehat{\mathrm{d}\sigma}(x_3-x_4)\widehat{\mathrm{d}\sigma}(x_4-x_1)
\end{equation}
is the kernel of a co-positive quartic form (positive on the cone of nonnegative functions) for some $c>0$. Evidently this would follow if $\widehat{k}_4$ were the kernel of a positive quartic form, as it did in the case $p=2$, although from this perspective it would seem harder to make use of a positivity assumption on $w$, should one be required.

\subsection{Reverse inequalities for \texorpdfstring{$p<1$}{p<1}}\label{Sect:rev}
In this section we make some remarks about the suggested inequality \eqref{p} and its variants in the hitherto excluded range $p<1$.
The first thing to observe is that \eqref{p} is easily seen to fail in general for $p<1$. To see this let $(C_j)$ be a (maximal) collection of disjoint $\delta$-caps on $\mathbb{S}^{n-1}$, and let $g_j=|C_j|^{-1/2}\mathbf{1}_{C_j}$. Taking $w$ to be the indicator function of a small ball centred at the origin in $\mathbb{R}^n$, we conclude from \eqref{p} and the local constancy of $|\widehat{g\mathrm{d}\sigma}|^2$ that
$$
\sum_j|C_j|^p=\sum_j|\widehat{g_j\mathrm{d}\sigma}(0)|^{2p}\sim\sum_j\Bigl(\int_{\mathbb{R}^n}|\widehat{g_j\mathrm{d}\sigma}|^2w\Bigr)^p\lesssim\|Xw\|_p^p\sim 1.
$$
Evidently this is absurd for $\delta$ sufficiently small if $p<1$.

It is conceivable that \eqref{p} holds \textit{in reverse} for $p<1$ when $(g_j)$ is \textit{complete} in $L^2(K)$ -- a suggestion supported by wavepacket heuristics, along with the observation that \eqref{p} is an identity when $p=1$ for complete $(g_j)$. The latter is immediate from Parseval's identity and Fubini's theorem since
\begin{equation}\label{pid}
\sum_j\int_{\mathbb{R}^n}|\widehat{g_j\mathrm{d}\sigma}|^2w=\int_{\mathbb{R}^n}\sum_j|\langle g_j,e_x\rangle_{L^2(K)}|^2w(x)\mathrm{d}x=|K|\int_{\mathbb{R}^n}w=\|Xw\|_{L^1(\omega\in K)},
\end{equation}
where $e_x(\xi):=e^{2\pi i x\cdot\xi}$ with $\xi\in\mathbb{S}^{n-1}$.
As further supporting evidence, the limiting case of \eqref{p} as $p\rightarrow 0$ is quickly seen to hold in reverse. To see this observe that the left-hand side becomes 
\begin{equation}\label{p0}
\#\left\{j:\int_{\mathbb{R}^n}|\langle g_j,e_x\rangle|^2w(x)\mathrm{d}x\not=0\right\},
\end{equation}
and so it suffices to show that this is finite only if $w=0$. If the expression \eqref{p0} is finite then for each $j$ in the complement of some finite set $J\subset\mathbb{Z}$ there is a null set $E_j\subset\mathbb{R}^n$ such that $\langle g_j,e_x\rangle=0$ for all $x\in \supp(w)\backslash E_j$, meaning that
$$
\{e_x:x\in\supp(w)\backslash E\}\subseteq\sp\{g_j:j\in J\},
$$
by the completeness of $(g_j)$.
This gives a contradiction if $\supp(w)$ is assumed to have positive measure since $E:=\cup_j E_j$ is null. 

Further modest evidence for such reverse bounds for complete orthonormal sequences may be found in the very special case where the underlying sphere is replaced with a hyperplane. To see this it is helpful to first revisit \eqref{onmtgen} and ask a similar question: for $0<p\leq 1$, under what conditions on a sequence of functions $(g_j)$ on a submanifold $S$ is it true that
\begin{equation}\label{onmtgenrev}
\sum_j\left(\int_{\mathbb{R}^n}|\widehat{g_j\mathrm{d}\sigma}|^2w\right)^p\gtrsim\|X_Sw\|_{L^p(\{(u,v)\in TS:u\in E\})}^p
\end{equation}
for some geometrically-meaningful set $E\subseteq S$ containing $$K=\bigcup_j\supp(g_j)?$$ For a hyperplane $S=\mathbb{R}^{n-1}\times\{0\}$ the inequality \eqref{onmtgenrev} is easily seen to hold for all complete orthonormal $(g_j)$ with $E=K$ and implicit constant $1$. Arguing as in the analogous forward claim (see Section \ref{Sect:gen}), matters are reduced to showing that
\begin{equation}\label{goodmorning}
\sum_j\left(\int_{\mathbb{R}^{n-1}}|\widehat{g}_j(v)|^2Xw(e_n,v)\mathrm{d}v\right)^p\geq |K|\|Xw(e_n,\cdot)\|_{L^p(\mathbb{R}^{n-1})}^p
\end{equation}
for all complete orthonormal sequences $(g_j)$ in $L^2(K)$, where $K\subseteq\mathbb{R}^{n-1}$.
To see this we set $\mathfrak{w}(v)=Xw(e_n,v)$ and use the completeness property 
$$\sum_{j}|\widehat{g}_j|^2\equiv |K|$$
to write
$$
|K|\|\mathfrak{w}\|_{L^p(\mathbb{R}^{n-1})}^p=\sum_j\int_{\mathbb{R}^{n-1}}\left(\mathfrak{w}(v)|\widehat{g}_j(v)|^2\right)^p|\widehat{g}_j(v)|^{2(1-p)}\mathrm{d}v,
$$
and then apply H\"older's inequality (with exponents $1/p, 1/(1-p)$) in the integral on the right-hand side to obtain
$$
|K|\|\mathfrak{w}\|_{L^p(\mathbb{R}^{n-1})}^p\leq\sum_j\left(\int_{\mathbb{R}^{n-1}}\mathfrak{w}(v)|\widehat{g}_j(v)|^2\mathrm{d}v\right)^p\left(\int_{\mathbb{R}^{n-1}}|\widehat{g}_j(v)|^2\mathrm{d}v\right)^{1-p}.
$$
This establishes \eqref{goodmorning} by Plancherel's theorem and the fact that $\|g_j\|_2=1$.

\begin{remark}[Local estimates]
For $0<p\leq 1$ one may easily obtain the local estimate
\begin{equation}\label{loca}
\sum_j\left(\int_{\mathbb{R}^n}|\widehat{g_j\mathrm{d}\sigma}|^2w\right)^p\gtrsim R^{-(1-p)(n-1)}\|Xw\|_{L^p(\{(\omega,v): \omega\in K, |v|\leq R\})}^p
\end{equation}
by attempting to follow the argument that led to \eqref{goodmorning}. To see this let
$\Phi$ be a function on $\mathbb{R}^n$ whose (restricted) Radon Transform $\mathcal{R}_0\Phi(\omega)$ belongs to $L^{1-p}(K)$, and use 
the completeness identity 
\begin{equation}\label{idagain}
\sum_j|\widehat{g_j\mathrm{d}\sigma}|^2\equiv |K|,
\end{equation}
to observe that 
\begin{eqnarray*}
    \begin{aligned}
        \int_{K}\int_{\langle\omega\rangle^\perp} &(Xw(\omega,v))^p\Phi(v)^{1-p}\mathrm{d}v\mathrm{d}\sigma(\omega)\\
&=|K|^{-p}\int_{K}\int_{\langle\omega\rangle^\perp}\left(\int_{\mathbb{R}}w(v+t\omega)\left(\sum_j|\widehat{g_j\mathrm{d}\sigma}(v+t\omega)|^2\right)\mathrm{d}t\right)^p\Phi(v)^{1-p}\mathrm{d}v\mathrm{d}\sigma(\omega)\\
&\leq |K|^{-p}\sum_j\int_{K}\int_{\langle\omega\rangle^\perp}\left(\int_{\mathbb{R}}w(v+t\omega)|\widehat{g_j\mathrm{d}\sigma}(v+t\omega)|^2\mathrm{d}t\right)^p\Phi(v)^{1-p}\mathrm{d}v\mathrm{d}\sigma(\omega).
    \end{aligned}
\end{eqnarray*}
Applying H\"older (with exponents $1/p, 1/(1-p)$) in $v$ and Fubini's theorem, this is bounded by
\begin{eqnarray*}
    \begin{aligned}
\sum_j\int_{K}\left(\int_{\langle\omega\rangle^\perp}\int_{\mathbb{R}}w(v+t\omega)|\widehat{g_j\mathrm{d}\sigma}(v+t\omega)|^2\mathrm{d}t\mathrm{d}v\right)^p&\mathcal{R}_0\Phi(\omega)^{1-p}\mathrm{d}\sigma(\omega)\\
&=\|\mathcal{R}_0\Phi\|_{L^{1-p}(K)}^{1-p}\sum_j\left(\int_{\mathbb{R}^n}|\widehat{g_j\mathrm{d}\sigma}|^2w\right)^p,
\end{aligned}
\end{eqnarray*}
up to the factor of $|K|^{-p}$.
The estimate \eqref{loca} follows from this upon setting $\Phi=\mathbf{1}_{B(0,R)}$.
\end{remark}
\begin{remark}[Possible entropic versions]
Since \eqref{p} is an identity at $p=1$ when $(g_j)$ is complete (see \eqref{pid}), one might expect there to be an entropic version of \eqref{p} in the case $K=\mathbb{S}^{n-1}$, accessible by differentiating the $L^p$ norms on both sides with respect to $p$ at $p=1$. In this context the weight $w$ is naturally taken to be a probability density, which makes $Xw$ a probability density on phase-space, and the sequence of ``expected intensities" $I=(I_j)$ given by $$I_j:=\mathbb{E}(|\widehat{g_j\mathrm{d}\sigma}|^2)=\int_{\mathbb{R}^n}|\widehat{g_j\mathrm{d}\sigma}|^2w$$ a probability density on $\mathbb{Z}$; strictly speaking one has to normalise $\mathrm{d}\sigma$ to be a probability measure here. If \eqref{p} were to hold with constant $1$ for $1\leq p\leq 2$, as it does at both extremes when $n\geq 3$ at least (this is \eqref{pid} and Theorem \ref{mainsphereintro} respectively), or indeed in reverse with constant $1$ for $p<1$, then it would follow that
$$
h(I)\geq h(Xw),
$$
where the entropy $h(f):=-\int f\log f$ is defined with respect to the appropriate measures on each side. We refer to \cite{BT} for examples of such arguments, along with some established properties of $h(Xw)$.
\end{remark}

\begin{remark}
Curiously the special case of the forward inequality \eqref{p} involving a \textit{single} input function -- see \eqref{single} --
and without the restriction on the X-ray norm on the right-hand side, may be seen to hold for all $p>0$, at least if one assumes that $w$ satisfies a suitable local constancy condition. This follows from the recent X-ray lower bounds in \cite{BT}, where it is shown that 
\begin{equation}\label{bt}
\|Xf\|_p\geq\|f\|_q
\end{equation}
whenever $f$ is nonnegative and $0<p,q\leq 1$ satisfy 
\begin{equation}\label{scaling}
\frac{1}{n}\left(1-\frac{1}{p}\right)=\frac{1}{n-1}\left(1-\frac{1}{q}\right).
\end{equation}
As this last relation defines a bijective map $p\mapsto q$ on $(0,1]$ it suffices to show that
$$\int_{\mathbb{R}^n}|\widehat{g\mathrm{d}\sigma}|^2w\lesssim\|g\|_2^2\|w\|_q$$ for all $0<q\leq 1$. However, this follows from the trivial $q=1$ case along with the estimate $\|w\|_1\lesssim\|w\|_q$; the latter following from the local constancy of $w$.
A local constancy assumption is natural here since $|\widehat{g\mathrm{d}\sigma}|^2$ is the Fourier transform of a compactly supported measure; see \cite{CIW}.
\end{remark}

\end{document}